\documentclass{amsproc}
\usepackage{latexsym, graphics, psfrag, amscd, amssymb}
\usepackage{xcolor} 
\usepackage{stmaryrd} 
\usepackage[all]{xy} 
\usepackage[normalem]{ulem} 

\newcommand{\fB}{\mathfrak{B}}

\newcommand{\cA}{\mathcal{A}}

\newcommand{\cC}{\mathcal{C}}
\newcommand{\cD}{\mathcal{D}}
\newcommand{\cF}{\mathcal{F}}
\newcommand{\cH}{\mathcal{H}}
\newcommand{\cL}{\mathcal{L}}
\newcommand{\cM}{\mathcal{M}}
\newcommand{\cO}{\mathcal{O}}

\newcommand{\cQ}{\mathcal{Q}}
\newcommand{\cT}{\mathcal{T}}
\newcommand{\cY}{\mathcal{Y}}
\newcommand{\cX}{\mathcal{X}}
\newcommand{\cZ}{\mathcal{Z}}

\newcommand{\bH}{\mathbb{H}}
\newcommand{\bL}{\mathbb{L}} 
\newcommand{\bZ}{\mathbb{Z}}
\newcommand{\bP}{\mathbb{P}}
\newcommand{\bR}{\mathbb{R}}
\newcommand{\bC}{\mathbb{C}}
\newcommand{\bQ}{\mathbb{Q}}
\newcommand{\RP}{\mathbb{RP}}

\newcommand{\GL}{\mathrm{GL}}
\newcommand{\SL}{\mathrm{SL}}
\newcommand{\MF}{\mathrm{MF}} 
\newcommand{\Spec}{\mathrm{Spec}}
 
\newcommand{\CY}{\mathrm{CY}}
\newcommand{\LG}{\mathrm{LG}}
\newcommand{\vir}{ {\mathrm{vir}} }
\newcommand{\rank}{\mathrm{rank}}
\newcommand{\Hom}{\mathrm{Hom}}

\newcommand{\Pic}{\mathrm{Pic}}
\newcommand{\dd}{\mathrm{d}}
\newcommand{\ch}{\mathrm{ch}}
\newcommand{\Ch}{\mathrm{Ch}}
\newcommand{\ev}{\mathrm{ev}}

\newcommand{\tot}{\mathrm{tot}}
\newcommand{\Eff}{\mathrm{Eff}}

\newcommand{\even}{\mathrm{even}}
\newcommand{\id}{\mathrm{id}}
\newcommand{\age}{\mathrm{age}}
\newcommand{\CR}{\mathrm{CR}}
\newcommand{\Crit}{\mathrm{Crit}}
\newcommand{\tdch}{\mathrm{tdch}}
\newcommand{\inv}{\mathrm{inv}}

\newcommand{\one}{\mathbf{1}}
\newcommand{\bft}{\mathbf{t}}
\newcommand{\be}{\mathbf{e}}
\newcommand{\bp}{\mathbf{p}}

\newcommand{\fg}{\mathfrak{g}} 
\newcommand{\forget}{\mathfrak{forget}}

\newcommand{\tG}{\widetilde{G}}
\newcommand{\tN}{\widetilde{N}}
\newcommand{\tX}{\widetilde{X}}

\newcommand{\tphi}{\widetilde{\phi}}

\newcommand{\hG}{\widehat{G}}

\newcommand{\lra}{\longrightarrow} 
\newcommand{\Mbar}{\overline{\cM}}

\newcommand{\bt}{{\bf t}}

\newtheorem{theorem}{Theorem}[section]
\newtheorem{lemma}[theorem]{Lemma}
\newtheorem{conjecture}[theorem]{Conjecture}
\newtheorem{proposition}[theorem]{Proposition}
\newtheorem{corollary}[theorem]{Corollary}

\theoremstyle{definition}
\newtheorem{definition}[theorem]{Definition}

\theoremstyle{remark}
\newtheorem{remark}[theorem]{Remark}

\numberwithin{equation}{section}

\begin{document}

\title[Open/closed correspondence and  extended  LG/CY correspondence]{Open/closed Correspondence and  Extended LG/CY Correspondence   for Quintic Threefolds} 
\author{Konstantin Aleshkin}
\address{Konstantin Aleshkin, Department of Mathematics, Columbia University, 2990 Broadway, New York, NY 10027, USA}
\curraddr{Kavli Institute for the Physics and Mathematics of the Universe (WPI), The University of Tokyo Institutes for Advanced Study, The University of Tokyo, Kashiwa, Chiba 277-8583, Japan}
\email{konstantin.aleshkin@ipmu.jp}

\author{Chiu-Chu Melissa Liu}
\address{Chiu-Chu Melissa Liu, Department of Mathematics, Columbia University, 2990 Broadway, New York, NY 10027}
\email{ccliu@math.columbia.edu}
\thanks{The second author was supported in part by NSF Grant DMS-1564497.}


\dedicatory{Dedicated to Professor Ezra Getzler on the occasion of his sixtieth birthday.}


\begin{abstract}
  We show that Walcher's disk potential for the quintic threefold can be represented as
  a central charge of a specific Gauged Linear Sigma Model which we call the
extended  quintic GLSM. This representation provides an open/closed correspondence
  for the quintic threefold since the central charge is a generating function of closed genus-zero GLSM invariants. 
We also explain how open Landau-Ginzburg/Calabi-Yau correspondence
and open mirror symmetry for the quintic are compatible with wall-crossing and mirror symmetry of the extended GLSM,
respectively.
\end{abstract}

\maketitle

\tableofcontents

\section{Introduction} 


Open/closed correspondence  proposed by Mayr~\cite{Ma} relates 
genus-zero A-model topological open strings on a toric Calabi-Yau threefold $X$ to 
genus-zero A-model topological closed strings on a toric Calabi-Yau fourfold $\tX$.
More precisely, the correspondence relates (i) genus-zero open Gromov-Witten (GW) 
invariants  counting holomorphic disks in a symplectic toric Calabi-Yau threefold $X$ with boundaries in
an Aganagic-Vafa Lagrangian A-brane  $L$,  and  (ii) genus-zero (closed) GW invariants counting holomorphic spheres in a symplectic toric 
Calabi-Yau fourfold $\tX$ determined by the pair $(X,L)$. The correspondence has been proved in full generality and also generalized
to toric Calabi-Yau 3-orbifolds by  S. Yu and the second named author \cite{LY22, LY}. Open string mirror symmetry relates
disk invariants of $X$ to relative periods of a holomorphic 3-form on the Calabi-Yau threefold $X^\vee$ mirror to $X$;
closed string mirror symmetry relates genus-zero GW invariants of $\tX$ to period integrals of a holomorphic 4-form on the
Calabi-Yau fourfold $\tX^\vee$ mirror $\tX$. There is a B-model open/closed correspondence which is 
compatible with the A-model open/closed correspondence and mirror symmetry \cite{LY}.  The extended Picard-Fuchs equations satisfied by the disk potential
(together with components of the $I$-function of $X$) can be identified with the 
Picard-Fuchs equations satisfied by components of the $I$-function of the toric Calabi-Yau fourfold $\tX$~\cite{Ma, LM, FL13, FLT22}.

Toric Calabi-Yau manifolds/orbifolds are non-compact. Open GW invariants in (i) and  closed GW invariants in (ii) are defined and computed by torus localization. 
In this paper we extend open/closed correspondence to compact Calabi-Yau threefolds. 
Solomon \cite{So} defined open GW invariants counting holomorphic disks in a compact symplectic manifold $(X,\omega)$ with boundaries in
a Lagrangian submanifold $L\subset X$ under the assumption that $L=X^\phi$ is the fixed locus of an anti-symplectic involution $\phi:X\to X$, 
is relatively spin, and  $\dim_\bR L = 2$ or  $3$. Solomon’s disk invariants coincide with Welschinger’s invariants counting real pseudo-holomorphic curves in the strongly positive case \cite{We05, We}.  When $X=X_5$  is the quintic Calabi-Yau threefold and $X^\phi=\bR X_5$ is the real quintic, a mirror formula for disk invariants of the pair $(X_5, \bR X_5)$ was conjectured by Walcher \cite{Wa07} and proved in \cite{PSW08}, see Section \ref{sec:CYopen}.  

It is well-known that components of the $I$-function of the quintic threefold  satisfy the degree 4
Picard-Fuchs differential equation~\eqref{eq:iPF}. Walcher's B-model disk potential $\cT^{\CY}$ 
satisfies a weaker degree 5 differential equation~\eqref{eq:ePF} together
with components of the $I$-function of $X_5$. The extended Picard-Fuchs equation \eqref{eq:ePF} is not the Picard-Fuchs equation of any Calabi-Yau variety.
In this paper we notice that Equation~\eqref{eq:ePF} is the equation satisfied
by components of the $I$-function of a particular Gauged Linear Sigma Model (GLSM) which we call the extended quintic GLSM, see Section \ref{sec:extended}. 
See~\cite{Wi93} for the introduction of GLSM in physics, \cite{FJR18,CFGKS,FK} for a mathematical theory
via algebraic geometry, and \cite{TX18, TX21, TX16, TX20, TX17, TX} for a mathematical theory via symplectic geometry. 

GLSM is a GW type enumerative theory that generalizes the
classical theory to the case where the target is a (not necessarily smooth)
critical locus of a regular function on a GIT quotient DM stack of a vector space.
In particular case where the critical locus is a smooth complete intersection DM stack, GLSM invariants  are GW
invariants of the critical locus up to a sign~\cite{CFGKS}.

We show that the disk potential  $\cT^{\CY}$ is equal to the
extended GLSM central charge~\cite{AL23} associated to a certain matrix factorization
up to a constant factor (Proposition \ref{prop:openHC}). 
Geometry of the extended GLSM depends on the choice of the
superpotential. We show that there is a natural one-parametric family of
superpotentials deforming the classical quintic superpotential. For the generic
value of the parameter the critical locus is a smooth DM stack which is a product
$X_5\times B\mu_2$  of the quintic threefold $X_5$ and the classifying space $B\mu_2$ of $\mu_2= \{ \pm 1\}$. 
In this case, the GLSM central charges~\cite{AL23} reproduce the components of the $I$-functions of $X_5$.
However, for the degenerate choice of the parameter the critical locus becomes larger and 
singular, and in this case there is an additional central charge corresponding
to $\cT^{\CY}$. 

GLSM setting is very convenient for studying the wall-crossing problems
in the target geometry such as Landau-Ginzburg/Calabi-Yau (LG/CY) correspondence.
For the quintic threefold it has been worked out by Chiodo-Ruan~\cite{CR10}.
Walcher~\cite{Wa07} performed analytic continuation of the disk potential $\cT^{\CY}$ in the CY phase
and obtained a disk potential $\cT^{\LG}$ in the LG phase. This leads to a mirror conjecture on  genus-zero open FJRW invariants, see Section \ref{sec:LGopen}. Genus-zero open FJRW  invariants have been constructed for the affine LG models 
$(x^r, \mu_r)$ \cite{BCT22, BCT}, $(x^r+ y^s, \mu_r\times \mu_s)$ \cite{GKT},  and (after the first version of this paper appeared in arXiv)  $(x_1^r + \cdots + x_m^r, \mu_r)$ \cite{TZ1, TZ2}; the case 
$(x_1^5 + \cdots + x_5^5, \mu_5)$ corresponds to the quintic threefold. 

In Section \ref{sec:central-charge-hybrid}, we show that the extended (open and closed) LG/CY correspondence corresponds to the usual wall-crossing
for closed invariants of the extended quintic GLSM~\cite{AL23}.  
In this case, the GIT stability condition $\zeta$ is a real number. The positive phase $\zeta>0$ corresponds to the CY phase, while
the negative phase $\zeta<0$ corresponds to the LG phase; they are separated by the wall $\zeta=0$. 
In Section \ref{sec:period-positive} and Section \ref{sec:period-negative},  we provide the mirror integral representations for the disk potentials $\cT^{\CY}$ and $\cT^{\LG}$, respectively, 
from the extended GLSM viewpoint. 

Buryak-Clader-Tessler \cite{BCT24} establish an open-closed correspondence between open $r$-spin theory 
and the closed extended $r$-spin theory constructed in \cite{BCT19}. In work in progress \cite{Ni},  S. Nill  
provides a different extension of genus-zero FJRW theory of $(x_1^5+\cdots + x_5^5, \mu_5)$ by integration of the fifth power of
Witten's top Chern class of the closed extended $5$-spin theory.

\subsection*{Outline} 
In Section~\ref{sec:review}, we review the LG/CY correspondence and the disk potential
for the quintic threefold. Section \ref{sec:GLSM} is an introduction to GLSM $I$-functions and central
charges following~\cite{AL23}. In particular, we review the Higgs-Coulomb correspondence
and wall crossing for central charges. The wall crossing is demonstrated on
the quintic example in Section~\ref{sec:LGCY0}. In Section~\ref{sec:extended} we introduce
the extended GLSM, compute $I$-functions and central charges
and show the open/closed correspondence for the quintic threefold.
Finally, in Section~\ref{sec:B} we explain the B-model side of the story. First, we
review the B-model representation of the disk potential following~\cite{Wa07}.
Then we provide the B-model of the extended GLSM and show how the quintic disk potential
appears as a usual period of the mirror extended GLSM.

\subsection*{Acknowledgments} We thank Kentaro Hori, Johanna Knapp, Sebastian Nill, Renata Picciotto, and Johannes Walcher for helpful discussion. We thank Sebastian Nill for explaining his work. We thank Gang Tian, Johannes Walcher, Guangbo Xu, Song Yu, and the anonymous referee for their comments on  early versions of this paper.  The second named author wishes to thank Ezra Getzler for numerous helpful conversations on GW theory, open GW theory, 
and mirror symmetry, as well as his encouragement, help, and support, over the past twenty years.

\section{LG/CY correspondence and mirror symmetry}  \label{sec:review}

\subsection{LG/CY correspondence}

An affine Landau-Ginzburg (LG) model is a pair $(W,G)$, where the superpotential
$W:\bC^N \to \bC$ is a non-degenerate, quasi-homogeneous
polynomial, and $G$ is an admissible subgroup of the finite abelian group
$G_W :=\{ \gamma\in (\bC^*)^N: W(\gamma x) = W(x)\}$.
 Fan-Jarvis-Ruan-Witten (FJRW) invariants of $(W,G)$ are virtual counts of sections of orbifold
line bundles over orbicurves which are solutions of the
(perturbed) Witten equations defined by the superpotential $W$. FJRW theory \cite{FJR07} can be viewed as a mathematical theory of A-model topological strings on affine LG models.

Gromov-Witten (GW) invariants are virtual counts of parametrized holomorphic curves in K\"{a}hler manifolds, or more generally
parametrized holomorphic orbicurves in K\"{a}hler orbifolds.  
GW theory can be viewed as a mathematical theory of A-model topological strings on K\"{a}hler manifolds/orbifolds.

Given an affine LG model $(W,G)$, the quasi-homogeneity of $W$ implies that the equation 
$W=0$ defines a hypersurface $X_W$ in a weighted
projective space, and the non-degeneracy of $W$ implies that
$X_W$ has at most quotient singularities. Suppose
that $X_W$ is Calabi-Yau in the sense that
$c_1(X_W)=0 \in H^2(X_W;\bQ)$. Then 
$\cX_W := [X_W/\tG]$ is a compact Calabi-Yau  orbifold, where $\tG$ is the quotient of $G$ by 
a cyclic subgroup $\langle J_W\rangle$ which acts trivially on $X_W$.  The Landau-Ginzburg/Calabi-Yau (LG/CY) correspondence relates  FJRW theory of the affine LG model $(W,G)$ and (orbifold) GW theory of the Calabi-Yau orbifold $\cX_W$.

\subsection{LG/CY correspondence for quintic threefolds} 
Chiodo-Ruan \cite{CR10} proved genus-zero LG/CY correspondence when 
$W= W_5: = x_1^5+\cdots + x_5^5$ is the Fermat quintic polynomial on $\bC^5$ and $G= \mu_5$ is the group of 5-th roots of unity, acting diagonally on  $\bC^5$. In this case $\tG$ is trivial, and  
$$
\cX_W= X_W = X_5 :=\{ W_5(x)=0\} \subset \bP^4
$$
is the Fermat quintic threefold, which is a non-singular projective Calabi-Yau threefold. 
Indeed, all smooth quintic threefolds (i.e. degree five hypersurfaces in $\bP^4$) have the same GW theory, so Chiodo-Ruan proved
genus-zero LG/CY correspondence for all (smooth) quintic Calabi-Yau threefolds.  
\subsubsection{GW invariants} 
Given any positive integer $d$, let $\Mbar_{0,0}(X_5, d)$ and $\Mbar_{0,0}(\bP^4,d)$ be moduli spaces of genus-0, $0$-pointed,
degree $d$ stable maps to $X_5$ and to $\bP^4$, respectively.  The genus-zero, degree $d$ GW invariant of $X_5$ is 
$$
N_d^{\CY}:= \int_{[\Mbar_{0,0}(X_5,d)]^\vir}1 = \int_{[\Mbar_{0,0}(\bP^4,d)] }e(E_d) \in \bQ
$$
where $E_d$ is an orbifold vector bundle over the compact complex orbifold $\Mbar_{0,0}(\bP^4,d)$ whose fiber over the moduli point $[f:C\to\bP^4]$ is 
$H^0(C, f^*\cO_{\bP^4}(1))$, and $e(E_d)= c_{5d+1}(E_d)$ is the Euler class of $E_d$; we have
$\rank E_d =\dim \Mbar_{0,0}(\bP^4,d)=5d+1$.  The genus-zero GW potential of $X_5$ is
$$
F^{\CY}_0(Q) :=\frac{5}{6} (\log Q)^3 + \sum_{d=1}^\infty N^{\CY}_d Q^d.  
$$

\subsubsection{The CY mirror theorem} 

The (small) $J$-function of GW theory of $X_5$ is
\begin{equation}\label{eq:J-CY}
J^{\CY}(Q,z) = z e^{ (\log Q) H/z}\left(1+\sum_{d=1}^\infty Q^d  \sum_{k=0}^3 \left(\int_{[\Mbar_{0,1}(X_5,d)]^\vir}\frac{\ev^*H^{3-k}}{z(z-\psi)} \right)\frac{H^k}{5}
\right)
\end{equation} 
where $H\in H^2(X_5)$ is the restriction of the hyperplane class in $H^2(\bP^4)$ 
and $\psi=c_1(\bL)$ is the first Chern class of a line bundle $\bL$ on 
$\Mbar_{0,1}(X_5,d)$ whose fiber over
the moduli point $[f:(C,x)\to X_5]$ is the cotangent line $T_x^*C$.
Let $D=Q\frac{\partial}{\partial Q}$. Then
$J^{\CY}(Q,z)$ can be expressed in terms of $F^{\CY}_0$:
\begin{equation}
\begin{aligned}
J^{\CY}(Q,z) =&  z  +  (\log Q)  H +  \frac{1}{5} DF^{\CY}_0(Q)   z^{-1}H^2 \\
& + \Big(\frac{1}{5} (\log Q) DF^{\CY}_0(Q) -\frac{2}{5}F^{\CY}_0(Q)  \Big)     z^{-2} H^3.
\end{aligned}
\end{equation}

The (small) $I$-function is given explicitly by 
\begin{equation}\label{eqn:I-CY}
I^{\CY}(q,z) = z e^{ (\log q) H/z} \Big( 1+ \sum_{d=1}^\infty q^d \frac{\prod_{m=1}^{5d}(5H+mz)}{\prod_{m=1}^d (H+mz)^5}  \Big)
= z \sum_{k=0}^3 I^{\CY}_k(q)\frac{H^k}{z^k}
\end{equation}
where $H^4=0$.  In particular,
$$
I^{\CY}_0(q)=  1+ \sum_{d=1}^\infty \frac{(5d)!}{(d!)^5} q^d, \quad  
\frac{I^{\CY}_1(q)}{I_0^{\CY}(q) } =  \log q + f^{\CY}(q), 
$$
where 
$$
f^{\CY}(q)= \frac{1}{I_0^{\CY}(q)}\sum_{d=1}^\infty q^d\frac{(5d)!}{(d!)^5}\Big(\sum_{m=1}^{5d}\frac{5}{m} -5\sum_{m=1}^d\frac{1}{m}\Big) \in 
q\bQ[\![q]\!].
$$

$I^{\CY}$ and $J^{\CY}$ are functions of one complex variable $q$ and $Q$ respectively. They take values in the 4-dimensional complex vector space
$$
\cH_{\CY} =\bC z  \oplus \bC H \oplus \bC z^{-1}H^2 \oplus \bC z^{-2}H^3
$$
equipped with the linear symplectic structure  
$$
(z^{1-k}H^k, z^{1-\ell}H^\ell) = 5(-1)^k \delta_{k+\ell,3}.
$$

The following mirror theorem was first proved by Givental \cite{Gi96} and Lian-Liu-Yau \cite{LLY97} and implies
the mirror conjecture on  $D^3F^{\CY}_0(Q)$ predicted by Candelas, de la Ossa, Green, Parkes \cite{CdGP}. 
\begin{theorem}[mirror theorem for the CY threefold $X_5$] \label{CYclosed}
\begin{equation}
J^{\CY}(Q,z) = \frac{I^{\CY}(q,z)}{I^{\CY}_0(q)} \quad \text{under the mirror map}\quad \log Q = \frac{I^{\CY}_1(q)}{I^{\CY}_0(q)}.
\end{equation}
\end{theorem} 
Ciocan-Fontanine and Kim \cite{CK14} introduce a family of $J$-functions $J^{\CY}_\epsilon$ where $\epsilon\in \bQ_{>0}$, by $\bC^*$-localization on $\epsilon$-stable quasimap graph spaces. The $I$-function and $J$-function correspond to 
the limits $\epsilon \to 0+$ and $\epsilon \to \infty$, respectively: 
$$
I^{\CY} = J^{\CY, 0+}, \quad J^{\CY} = J^{\CY, \infty}.
$$
Theorem \ref{CYclosed}  is a consequence of their quasimap wall-crossing formula \cite{CK14, CK20}. This is also known as $\epsilon$-wall-crossing. 

\subsubsection{FJRW invariants} 
A genus-$g$, $n$-pointed stable 5-spin curve is a triple $((C,z_1,\ldots, z_n), \cL, \rho)$, where
$(C,z_1,\ldots,z_n)$ is a genus-$g$, $n$-pointed twisted stable curve\footnote{In particular, $C$ is a Deligne-Mumford curve
with at most nodal singularities and is a scheme away from the marked points $z_i$ and nodes.}, $\cL$ is a line
bundle on $C$ which defines a representable morphism from $C$ to $B\bC^*$, and 
$\rho: \cL^{\otimes 5} \to \omega_C^{\log} = \omega_C(z_1+\cdots +z_n)$ is an isomorphism. Given 
$m_1,\ldots,m_n\in \{1,2,3,4\}$, let 
$\Mbar_{g, (m_1,\ldots,m_n)}^{1/5}$ denote the moduli of genus-$g$, $n$-pointed stable 5-spin curves $((C,z_1,\ldots,z_n), \cL,\rho)$
such that $z_i = B\mu_5$ and $L_{z_i} = (T_{z_i} C)^{\otimes m_i}$ for $i=1,\ldots,n$. Then $\Mbar_{g,(m_1,\ldots,m_n)}^{1/5}$ is 
nonempty if and only if
$$
2g-2 + \sum_{i=1}^n(1-m_i) \in 5\bZ
$$
in this case, it is a proper smooth Deligne-Mumford stack of dimension $3g-3+n$.  In particular, 
$$
M_n: =\Mbar_{0, ( \underbrace{2,\ldots,2}_n )}
$$
is nonempty if and only if $(n-3)/5 \in \bZ_{\geq 0}$, and in this case 
there is a rank $(n-3)/5$ vector bundle $E_n$ over
$M_n$ whose fiber over the moduli point $( (C,z_1,\ldots,z_n), \cL,\rho)$ is $H^1(C,\cL)$.  The genus-0, $n$-pointed primary
FJRW invariants of the affine LG model $(W_5, \mu_5)$ are
$$
\theta_n := \begin{cases}
\displaystyle{ \int_{[M_n]} e(E_n)^5,}  & \text{if } \displaystyle{ \frac{n-3}{5}\in \bZ_{\geq 0} }; \\
0, &\text{otherwise}.  
\end{cases} 
$$
Let 
$$
F^{\LG}_0(\tau):= \sum_{n=0}^\infty \theta_n \frac{\tau^n}{n!} = \sum_{\ell=0}^\infty \theta_{3+5\ell} \frac{\tau^{3+5\ell}}{(3+5\ell)!}
$$

\subsubsection{The LG mirror theorem} \label{LGclosed}
The state space $H_{\LG}$ of FJRW theory of $(W_5, \mu_5)$ is isomorphic
to $H^*(X_5;\bC)$ as a graded vector space over $\bC$.  In particular, 
for $k\in \{0,1,2,3\}$, 
$$
H^{2k}_{\LG} \cong H^{2k}(X_5;\bC) =\bC H^k \cong \bC. 
$$
The (small) $J$-function of FJRW theory of $(W_5, \mu_5)$ is
\begin{equation}\label{eq:J-CY}
J^{\LG}(\tau,z) = z \Big(\phi_0 + \sum_{n=1}^\infty \frac{\tau^n}{n!}
\sum_{k=0}^3 \big\langle \frac{\phi_{3-k}}{z-\psi}, \phi_0, 
\underbrace{ \phi_1,\ldots, \phi_1}_{n\text{ times}} \big\rangle_{0,n+2} \phi_k \Big)
\end{equation} 
where $\phi_k$ is a specific $\bC$-basis of $H^{2k}_{\LG}\cong \bC$.
The function $J^{\LG}(\tau,z)$ can be expressed in terms of the generating function  $F_0^{\LG}(\tau)$: 
\begin{equation}\label{eqn:J-F0}
J^{\LG}(\tau,z) = z \phi_0  +  \tau \phi_1 +  \frac{\partial F^{\LG}_0}{\partial \tau}(\tau)  z^{-1}\phi_2 + 
\Big(\tau \frac{\partial F^{\LG}_0}{\partial \tau} -2 F^{\LG}_0(\tau) \Big)     z^{-2} \phi_3.
\end{equation}
The (small) $I$-function $I^{\LG}(t,z)$ is given explicitly by 
\begin{equation}\label{eqn:I-LG}
I^{\LG}(t,z) =z\sum_{d=0}^\infty \frac{t^{d+1}}{d!}\frac{ \Gamma(\frac{d+1}{5})^5 }{ \Gamma( \{ \frac{d+1}{5}\} )^5 }
\frac{\phi_{5 \{ \frac{d}{5}\} }  }{ z^{ 5 \{\frac{d}{5}\} }  }= 
z\sum_{k=0}^3 I_k^{\LG}(t) \frac{\phi_k}{z^k}.
\end{equation}
\begin{eqnarray*}
I^{\LG}_0(t) &=&  t+ \sum_{\ell=1}^\infty \frac{t^{5m+1}}{(5m)!} \frac{\Gamma(\frac{1}{5}+m)^5}{\Gamma(\frac{1}{5})^5},  \\
 \frac{I^{\LG}_1(t)}{I_0^{\LG}(t) } &=&  \frac{t^{2}}{I_{0}^{\LG}(t)}\Big(1 +\sum_{m=1}^\infty \frac{t^{5m}}{(5m+1)!} \frac{\Gamma(\frac{2}{5}+m)^5}{\Gamma(\frac{2}{5})^5}\Big)
= t(1 + f^{\LG}(t) )
\end{eqnarray*}
where $f^{\LG}(t) \in t^5 \bQ[\![t^5 ]\!]$. 

$I^{\LG}(t,z)$ and $J^{\LG}(\tau,z)$ are functions of one (complex) variable,  taking values in the 4-dimensional complex vector space
$$
\cH_{\LG} =\bC z\phi_0 \oplus \bC \phi_1 \oplus \bC z^{-1}\phi_2 \oplus \bC z^{-2}\phi_3
$$
equipped with the linear symplectic structure  
$$
(z^{1-k}\phi_k , z^{1-\ell}\phi_\ell) = (-1)^k \delta_{k+\ell,3}.
$$

The results in \cite{CR10} imply the following mirror theorem, which provides a mirror formula of $F_0^{\LG}(\tau)$. 
\begin{theorem}[mirror theorem for the LG model $(W_5,\mu_5)$]
$$
J^{\LG}(\tau,z) = \frac{I^{\LG}(t,z)}{I^{\LG}_0(t)} \quad \text{under the mirror map}\quad \tau = \frac{I^{\LG}_1(t)}{I^{\LG}_0(t)}. 
$$
\end{theorem}
Ross-Ruan \cite{RR17} introduce a family of $J$-functions $J^{\LG,\epsilon}$ where $\epsilon\in \bQ_{>0}$, 
which interpolate the $I$-function and the $J$-function: 
$$
I^{\LG} = J^{\LG, 0+}, \quad J^{\LG} = J^{\LG, \infty}.
$$
Theorem \ref{LGclosed}  follows from their wall-crossing formula relating $J^{\LG, \epsilon_1}$ and $J^{\LG, \epsilon_2}$,
so it can be viewed as a version of $\epsilon$-wall-crossing. 

\subsubsection{Picard-Fuchs equation, analytic continuation, and symplectic transform}
Let $\theta = q\frac{d}{dq}$. Then $\{ I_k^{\CY}(q): k=0,1,2,3\}$ is a basis of the space of solutions to the Picard-Fuchs equation
\begin{equation}\label{eq:iPF}
\theta^4 \omega = 5q(5\theta+1)(5\theta+2)(5\theta+3)(5\theta+4)\omega.
\end{equation} 
Let $\theta_- = t \frac{d}{dt}$. Then $\{ I_k^{\LG}(t): k=0,1,2,3\}$ is a basis of the space of solutions to
\begin{equation}\label{eq:PF-LG}
(\theta_-)^4 \omega = 5^5 t^{-5}(\theta_- -1)(\theta_- - 2)(\theta_- - 3)(\theta_- -3)\omega.
\end{equation} 
Note that \eqref{eq:iPF} and \eqref{eq:PF-LG} are related by a simple change of variable 
$q=t^{-5}$, $\theta_- = -5\theta$.  

The Picard-Fuchs equation \eqref{eq:iPF} is satisfied by periods 
$\int_{\gamma}\Omega$ where $\Omega$ is a holomorphic 3-form on the mirror quintic $\cQ$
and $\gamma\in H_3(\cQ;\bZ)\cong \bZ^4$ \cite{CdGP}. See Section \ref{sec:B-model} for more details.  

The $I$-functions $I^{\CY}(q,z)$ and $I^{\LG}(t,z)$ are related by
change of variable $q=t^{-5}$, analytic continuation, and an isomorphism of complex symplectic vector spaces: 
$$
\phi: \cH_{\CY} \stackrel{\cong}{\lra} \cH_{\LG}. 
$$
The LG/CY correspondence of quintic threefolds is a consequence of {\em global mirror symmetry}
over the one-dimensional complex moduli $\bC_t$ of the mirror quintic $\cQ$, or equivalently the  complexified 
K\"{a}hler moduli of the quintic threefold $X_5$. 

\subsection{Open mirror symmetry}
Open GW invariants are virtual counts of parametrized holomorphic curves in K\"{a}hler manifolds with Lagrangian boundary conditions.
Open GW theory can be viewed as a mathematical theory of A-model topological open strings on K\"{a}hler manifolds.

Open FJRW invariants are virtual counts of sections of orbifold
line bundles over orbifold Riemann surfaces with boundaries, where the sections satisfy certain boundary conditions. Open FJRW theory can be viewed as a mathematical theory of A-model topological open
strings on affine LG models.

\subsubsection{The CY open mirror theorem} \label{sec:CYopen} 
The Fermat quintic $X_5 =\{ W_5(x)=0\}\subset \bP^4$ is defined over $\bR$. Let 
$\RP^4\subset \bP^4$ be the real projective space, and let  $\bR X_5 :=\RP^4\cap X_5$ be the real quintic, which is
diffeomorphic to $\RP^3$. The real quintic $\bR X_5$ is the fixed locus of an anti-holomophic and anti-symplectic involution on the K\"{a}hler manifold $X_5$, so it is a totally real Lagrangian submanifold of $X_5$.  The A-model disk potential $F^{\CY}_{0,1}$ of the pair $(X_5, \bR X_5)$ is 
$$
F^{\CY}_{0,1}(Q) =\sum_{\substack{ d \in \bZ_{> 0} \\ d \text{ odd } }}  N_d^{\text{disk}} Q^{d/2}
$$
where  $N_d^{\text{disk}}\in \bQ$ is the open GW invariant which counts degree $d$ holomorphic disks in $X_5$ bounded by $\bR X_5$. Near the maximal unipotent monodromy (MUM) point $q=0$ in the complex moduli of $\cQ$ (which corresponds to the CY/geometric phase of the K\"{a}hler moduli of $X_5$ under mirror symmetry),  the  B-model disk potential is given explicitly by
\begin{equation}\label{eq:B-model-disk}
\cT^{\CY}(q) = 2 \sum_{\substack{ d\in \bZ_{> 0} \\ d \text{ odd}} }\frac{(5d)!!}{(d!!)^5} q^{d/2} 
\end{equation}
Note that $\cT^{\CY}(q)= 30\sqrt{q}(1+ h(q))$ where $h(q)\in q\bQ[\![q ]\!]$. 

The following mirror formula for the A-model disk potential $F^{\CY}_{0,1}(Q)$ was conjectured by Walcher \cite{Wa07} and proved by
Pandharipande-Solomon-Walcher \cite{PSW08}. 
\begin{theorem}[Open mirror theorem for $(X_5, \bR X_5)$] \label{CYopen} 
$$
F^{\CY}_{0,1}(Q) = \frac{\cT^{\CY}(q)}{I^{\CY}_0(q)} \quad \text{\em under the mirror map} \quad \log Q= \frac{I^{\CY}_1(q)}{I^{\CY}_0(q)}.
$$
\end{theorem}

In the following table, we compare notations in this paper with those in \cite{Wa07} and \cite{PSW08}.
 \begin{center}
 \renewcommand{\arraystretch}{1.3}
 \begin{tabular}{|c|c|c|c|c|}
\hline 
This paper & $q$ & $\cT^{\CY}(q)$ & $Q$ & $F^{\CY}_{0,1}(Q)$ \\ \hline
\cite{Wa07} & $z$ & $30\tau(z)$ & $q=e^{2\pi i t}$ & $\hat{\tau}(q)$ \\  \hline 
\cite{PSW08}  & $e^t$ & $J(t)$ &  $e^T$ & $\cF^{disk}(T)$ \\  \hline
\end{tabular}
\renewcommand{\arraystretch}{1}
\end{center}

\smallskip

\subsubsection{The inhomogeneous and extended Picard-Fuchs equations} 
Let 
\begin{equation}\label{eqn:PF-operator}
\cL  := \theta^4 - 5q (5\theta +1)(5\theta+2)(5\theta+3)(5\theta +4). 
\end{equation}
Then the Picard-Fuch equation \eqref{eq:iPF} can be rewritten as
\begin{equation}
\cL\omega =0
\end{equation}
The B-model disk potential $\cT^{\CY}(q)$ satisfies the inhomogeneous Picard-Fuchs equation
\begin{equation} \label{eqn:inhomogeneous} 
\cL \cT = \frac{15}{8} \sqrt{q}.
\end{equation}
Therefore, 
\begin{equation}
(2\theta-1)\cL\cT =0
\end{equation}
which is equivalent to the extended Picard-Fuchs equation
\begin{equation}\label{eq:ePF}
(2\theta-1)\cL \cT = 5q (2\theta+1)(5\theta +1)(5\theta+2)(5\theta+3)(5\theta +4)\cT.
\end{equation}
The space of solutions to \eqref{eq:ePF} is spanned by $\{ I_k^{\CY}: k=0,1,2,3\}$ and $\cT^{\CY}$. 
Combining Theorem \ref{CYclosed} (closed mirror theorem) and Theorem \ref{CYopen} (open mirror theorem), we obtain a version
of {\em extended} mirror symmetry \cite{Wa08} including both closed string and open string sectors for quintic CY threefolds.

\subsubsection{The LG open mirror conjecture}  \label{sec:LGopen} 
By \cite[Section 3]{Wa07}, the analytic continuation of $\cT^{\CY}(q)$ to the orbifold point $q=\infty$  in the complex moduli of $\cQ$ (which corresponds to the LG phase of the K\"{a}hler moduli of $X_5$  under mirror symmetry) is
\begin{equation} \label{eq:LGCY1}
\cT^{\LG} + \cT_c
\end{equation}
where
\begin{eqnarray*}
\cT^{\LG}  &=& -\frac{2}{3}\sum_{m=0}^\infty \frac{\Gamma(-3/2-5m)}{\Gamma(-3/2)}\frac{\Gamma(1/2)^5}{\Gamma(1/2-m)^5} q^{-(m+1/2)}.  \\
&=& - 2 \sum_{\substack{d >0 \\ d \ \mathrm{odd}} } \frac{(d!!)^5}{d^5(5d-2)!!} q^{-d/2} \\
&=& - 2 \sum_{\ell=0}^\infty \frac{((2\ell+1)!!)^5}{(2\ell+1)^5(10\ell+3)!!} t^{\frac{5}{2} +5\ell},
\end{eqnarray*}
and $\cT_c$ is a linear combination of $\{ I_k^{\LG}(t): k=0,1,2,3\}$
\begin{equation}
  \label{eq:1}
  \begin{aligned}
    \cT_{c} = -\sqrt{-1}\frac{\pi^{3}}{10}\sum_{m=1}^{4}\frac{e^{4\pi\sqrt{-1}m/5}}{\Gamma(1-m/5)^{5}\cos(\pi m/5)} I^{\LG}_{m-1}(t).
  \end{aligned}
\end{equation}
In particular, $\cT^{\LG} = -2/3t^{5/2} + O(t^{5})$.
This leads to the following mirror conjecture (which is implicit in \cite{Wa09}). Let $F_{0,1}^{\LG}(\tau)$ be the (yet-to-be-defined) generating function
of genus-zero open FJRW invariants of $(W_5,\mu_5)$. 
\begin{conjecture}[open mirror conjecture for $(W_5, \mu_5)$]\label{LGopen}
$$
F^{\LG}_{0,1}(\tau) = \frac{\cT^{\LG}(t)}{I^{\LG}_0(t)} \quad \text{\em under the mirror map} \quad \tau = \frac{I^{\LG}_1(t)}{I^{\LG}_0(t)}.
$$
\end{conjecture}

\section{Gauged Linear Sigma Models} \label{sec:GLSM}
\subsection{Input data}
The input data of a  GLSM is a 5-tuple $(V, G, \bC^{*}_{R}, W, \zeta)$, where:
\begin{enumerate}
  \item (linear space) $V$ is a finite dimensional complex vector space.
  \item (gauge group) $G \subset \GL(V)$ is a complex reductive Lie group.
  \item ($R$ symmetries) $\bC^{*}_{R} \subset \GL(V)$ is a 1-parametric subgroup that commutes with $G$, and 
  $G \cap \bC^{*}_R$ is finite. Therefore $G\cap \bC^*_R \simeq \mu_{r}$ for some $r \in \bZ_{>0}$, where $\mu_r$ is the group of $r$-th roots of unity. 
  \item (superpotential) $W \: : \; V \to \bC$ is a polynomial function that is $G$-invariant and has weight $r$ with respect to the grading defined by
        the action of $\bC^{*}_{R}$ on $V$: $W(t\cdot x) =t^r W(x)$ for $t\in\bC^*_R$ and $x\in V$. 
  \item (stability condition) $\zeta \in \hG:=\Hom(G,\bC^*)$ is a stability condition for the geometric invariant theory (GIT) quotient stack
   $[V \sslash_{\zeta}  G] = [V^{ss}(\zeta)/G]$. 
\end{enumerate}
The GIT quotient stack $\cX_\zeta := [V\sslash_{\zeta} G]$ is a smooth Artin stack with trivial generic stabilizer (since $G$ acts faithfully on $V$). We assume that $V^{ss}(\zeta)=V^s(\zeta)$, so that $\cX_\zeta$ is an orbifold, i.e., a smooth Deligne-Mumford (DM) stack with trivial generic stabilizer.  The GIT quotient variety $X_\zeta:= V\sslash_{\zeta} G = V^{ss}(\zeta)/G$ is a good moduli of the stack $\cX_\zeta$. It is equipped with  a polarization $L_\zeta$ determined by the $G$-character $\zeta$, and is projective over its affinization $V/_{\mathrm{aff}}G := \mathrm{Spec}\left(H^0(V, \cO_V)^G\right)$.  The superpotential
$W \in H^0(V,\cO_V)^G$ defines a regular function  $w: V/_{\mathrm{aff}} G \to \bC$. Let $w_\zeta: \cX_\zeta \to \bC$ denote the pullback of $w$ under  $\cX_\zeta \to X_\zeta \to V/_{\mathrm{aff}} G$. Let 
$$
\cZ_\zeta^w := Z(dw_\zeta) 
$$
be the critical locus of $w_\zeta$. 

The mathematical theory of the GLSM developed by Fan-Jarvis-Ruan \cite{FJR18} can be viewed as a mathematical theory of A-model topological strings on the LG model $(\cX_\zeta, w_\zeta)$. FJRW theory \cite{FJR07} corresponds to the special case where $G$ is a finite abelian group,  $\cX_\zeta = [V/G]$ and $X_\zeta = V/_{\mathrm{aff}}G$.   Fan-Jarvis-Ruan \cite{FJR18} define GLSM invariants in the narrow sector via Kiem-Li cosection localized virtual cycle \cite{KL13},
generalizing construction in \cite{CLL} in the affine case.  See \cite{KL20, CKL} for construction of FJRW invariants in both broad and narrow sectors via cosection localization.  Favero-Kim \cite{FK} define general GLSM invariants via virtual factorization, generalizing construction of Polishchuk-Vaintrob \cite{PV16} in the affine case and construction of Ciocan-Fontanine, Favero, Gu\'{e}r\'{e}, Kim, Shoemaker \cite{CFGKS} in the convex hybrid case. 

Tian-Xu provides a mathematical theory of the GLSM in the geometric phase via symplectic geometry. In 
 \cite{TX17} and \cite{TX}, $V$ is a non-compact K\"{a}hler manifold and $W:V\to \bC$ is a $G$-invariant holomorphic function.
 
\subsection{LG quasimaps} \label{sec:LGquasimaps}
Let $\Gamma\subset \GL(V)$ be the subgroup generated by $G$ and $\bC_R^*$, and let
$$
\bC_\omega^* := \Gamma/G = \bC_R^*/(G\cap \bC_R^*) = \bC_R^*/\mu_r \simeq \bC^*.
$$
We have a short exact sequence of groups
$$
1\to G\lra \Gamma \stackrel{\chi}{\lra}\bC_\omega^*\to 1. 
$$

GLSM invariants are virtual counts of LG quasimaps. Roughly speaking, an LG quasimap to the given GLSM is a  rational map 
from a twisted curve  $(C,z_1,\ldots,z_n)$  to $\cX_\zeta= [V^{ss}_G(\zeta)/\Gamma]$ which extends to a morphism $f:C \to [V/\Gamma]$  such that the diagram
\begin{equation} \label{eqn:LGmap} 
\xymatrix{
 & [V/\Gamma] \ar[d] \\
C \ar[ur]^f \ar[r]^{P \quad} \ar[dr]_{\omega^{\log}_C} & B\Gamma = [\bullet/\Gamma] \ar[d]^{B\chi}  \\
& B\bC^*_\omega 
}
\end{equation}
commutes, and satisfies some stability conditions depending on $\epsilon\in \bQ_{>0}$ and a so-called {\em good lift} $\tilde{\zeta}$
which is  a $\Gamma$-character $\tilde{\zeta}:\Gamma\to \bC^*$ such that $\tilde{\zeta}|_G = \zeta$ and
$V_\Gamma^{ss}(\tilde{\zeta})= V_G(\zeta)$.

For a fixed domain $C$, the upper half of \eqref{eqn:LGmap} is equivalent to a principal $\Gamma$-bundle
$P\to C$ together with a $\Gamma$-equivariant map $\tilde{f}: P\to V$, or equivalently, a principal
$\Gamma$-bundle $P\to C$ together with a section $u$ of the vector bundle
$P\times_\rho V$ associated to the (faithful) representation $\rho: \Gamma\to \GL(V)$. The lower
half of \eqref{eqn:LGmap} says that the line bundle $P\times_{\chi} \bC$ associated
to the $\Gamma$-character $\chi:\Gamma\to \bC^*_\omega$ is isomorphic to  the log canonical line
bundle $\omega^{\log}_C = \omega_C(z_1+\cdots+z_n)$ of the marked curve $(C,z_1,\ldots,z_n)$.  The degree of $f: C\to [V/G]$ is an element in 
$$
H_2([V/\Gamma];\bQ) \simeq H_2(B\Gamma;\bQ)\simeq H_2(BG;\bQ)\oplus H_2(B\bC_\omega^*;\bQ). 
$$
where the first component is $d \in  H_2(BG;\bQ)\simeq \Hom(\hG,\bQ)$ and the second component is $\deg(\omega_C^{\log}) = 2g-2+n \in H_2(\bC_\omega^*;\bQ)=\bQ$. We call $d\in \Hom(\hG,\bQ)$ the degree of the quasimap. 

More precisely, a genus-$g$, $n$-pointed, degree $d$  LG quasimap to the given GLSM is a 4-tuple $((C,z_1,\ldots,z_n), P, \kappa, u)$, where $(C,z_1,\ldots, z_n)$ is a genus-$g$, $n$-pointed twisted curve, $P$ is principal $\Gamma$-bundle on $C$ which defines a representable morphism $C\to B\Gamma$, $\kappa: P\times_{\chi}\bC \stackrel{\simeq}{\lra} \omega^{\log}_C$ is an isomorphism
of line bundles on $C$, and $u$ is a section of the vector bundle $P\times_{\rho} V$ such that the base locus
$$
B:= u^{-1}(P\times_\rho V^{us}_G(\zeta)), 
$$
where $V^{us}_G(\zeta)=V- V^{ss}_G(\zeta)$ is the unstable locus, is a finite set away from nodes and marked
points $z_1,\ldots, z_n$.  

Let $\tilde{\zeta}$ be a good lift, and let $\epsilon \in \bQ_{>0}$. An LG quasimap is $\epsilon$-stable with respect to $\tilde{\zeta}$  if
the $\bQ$-line bundle $\omega_C^{\log}\otimes (P\times_{\tilde{\zeta}}\bC)^{\epsilon}$ is ample
and $\epsilon \ell(b)\leq 1$ for any point $b$ in the base locus, where $\ell(b)$ is the length defined by $\tilde{\zeta}$.  

The GLSM is called abelian if the gauge group $G$ is abelian. In this case,  we may choose coordinates $x_1,\ldots,x_N$ on $V\simeq \bC^N$ such that
both $G$ and $\bC_R^*$ act diagonally on $V$,  by characters $D_i \in \hG$
and  $c_i \in  \widehat{\bC}_R^*=\bZ$, respectively. Then $\Gamma$ acts on $V$ diagonally by characters $\tilde{D}_i\in \widehat{\Gamma}$ with $\tilde{D}_i|_G = D_i$ and $\tilde{D}_i|_{\bC^*_R} = c_i$.  
The $G$-characters $D_i$ are called gauged charges, and $R_i:= 2c_i/r \in \bQ$ are called $R$-charges in 
the physics literature. The GIT quotient stack $\cX_\zeta$ is a smooth toric DM stack \cite{BCS} and its coarse moduli 
space $X_\zeta$ is a semi-projective simplicial toric variety.  In this case $u = (u_1,\ldots, u_N)$ where $u_i$ is a section of the line bundle
$\cL_i := P\times_{\tilde{D}_i}\bC$ associated to the $\Gamma$-character $\tilde{D}_i$, i.e., 
$u_i\in H^0(C,\cL_i)$.  We have
$$
\deg(\cL_i) =\langle D_i, d\rangle + \frac{R_i}{2}(2g-2+n) \in \bQ. 
$$
where $\langle D_i, d\rangle\in \bQ$ is the pairing between $D_i \in \hG$ and $d \in \Hom(\hG,\bQ)= H_2(BG;\bQ)$, and $\displaystyle{ \frac{R_i}{2} }(2g-2+n)$ is the pairing between  $\displaystyle{ \frac{R_i}{2} =\frac{c_i}{r} }\in  H^2(B\bC_\omega^*;\bQ)\simeq\bQ$ and $2g-2+n\in H_2(B\bC_\omega^*;\bQ) \simeq \bQ$.

\subsection{LG loop spaces}
 Orbifold quasimap theory \cite{CCK} can be viewed as a mathematical theory of the GLSM without a superpotential. The $I$-function in orbifold quasimap theory can be obtained by $\bC^*$ localization
on stacky loop spaces. We now assume the gauge group $G$ is abelian, and define LG loop spaces which are analogues of the stacky loop spaces, following \cite{AL23}.

Given any positive integer $a$, let $\bP[a,1]=[(\bC^2\setminus\{0\})/\bC^*]$, where $\bC^*$ acts on $\bC^2$ by weights $(a,1)$. Let
$\infty:= [1,0] \in \bP[a,1]$ be the image of $(1,0)$ under the projection $\bC^2\setminus\{0\} \to \bP[a,1]$.  
Then $\bP[a,1] = \bC\cup \{\infty\}$. The stabilizer of $(1,0)$ is $\mu_a =\{ t\in \bC^*\mid t^a=1\} \subset \bC^*$, so 
$\infty$ is isomorphic to the classifying space $B\mu_a$  of $\mu_a$. 
For any $m\in \bZ$ let $\cO_{\bP[a,1]}(m)$ denote the line bundle over $\bP[a,1]$ with total space
$[ ((\bC^2-0)\times \bC)/\bC^*$ where $\bC^*$ acts by weights $(a,1,m)$.  Define
$$
H^p\Big(\bP^1, \cO\big(\frac{m}{a}\big)\Big) := H^p\big(\bP[a,1], \cO_{\bP[a,1]}(m)\big).
$$  
Note that this is well-defined because the right hand side depends only on $m/a$. 

Given $d\in \Hom(\hG, \bQ)$, define
$$ 
V_d := \bigoplus_{i=1}^N H^0\big(\bP^1,\cO(\langle D_i,d\rangle - \frac{R_i}{2})\big),\quad
W_d := \bigoplus_{i=1}^N H^1\big(\bP^1, \cO(\langle D_i,d\rangle - \frac{R_i}{2})\big),
$$ 
where the $R$-charges $R_i$ are determined by the input data of the abelian GLSM (see the last paragraph
of Section \ref{sec:LGquasimaps}). Let $G$ act on the summand $H^p\big(\bP^1, \cO(\langle D_i, d\rangle -R_i/2) \big)$ by the $G$-character $D_i$. 
The degree $d$ LG loop space $L_d$  and the obstruction bundle $E_d \to L_d$ are 
$$
L_d:= [V_d\sslash_\zeta G] = [ (V_d)^{ss}_G(\zeta)/G], \quad
E_d := [ \left( (V_d)^{ss}_G(\zeta) \times W_d\right)/G],
$$
We define the set of effective classes for the stability condition $\zeta\in\hG$ to be
$$
\Eff_\zeta =\{ d \in \Hom(\hG, \bQ): [V_d\sslash_\zeta G] \text{ is nonempty}\}.
$$
$L_d$ is a smooth toric DM stack (whose generic stabilizer might be non-trivial), and $E_d$ is a toric vector bundle over $L_d$.  The action of $\bC_{\mathsf{q}}^*\simeq \bC^*$ on $\bP^1$ induces a $\bC_{\mathsf{q}}^*$-action on $L_d$, and $E_d$ is a $\bC^*_{\mathsf{q}}$-equivariant vector bundles over $L_d$.  The virtual tangent bundle of $L_d$ is
$$
T^{\vir}_{L_d} = T_{L_d} - E_d. 
$$

There is a map $\ev_\infty: L_d\to I[V/G]$ given by evaluation at $\infty$. Let $L_d^\circ\subset L_d$ be the open substack
which is the preimage of  $I\cX_\zeta = I[V_G^{ss}(\zeta)/G]\subset I[V/G]$. Then $\bC_{\mathsf{q}}^*$ acts on $L_d^\circ$.
We have a disjoint union of connected components: 
$$
I\cX_{\zeta} =\bigsqcup_{v\in \mathrm{Box}_\zeta} \cX_{\zeta, v}
$$
where $\mathrm{Box}_\zeta$ is a finite set defined in \cite{BCS}.  Each element $v\in \mathrm{Box}_\zeta$ corresponds to a unique element $g(v)\in G$ such that
 $V_G^{ss}(\zeta)^{g(v)} :=\{ x \in V^{ss}_G(\zeta): g(v) \cdot x = x \}$ is non-empty, and 
 $$
 \cX_{\zeta,v} = [V^{ss}_G(\zeta)^{g(v)}/G].
 $$
 Let $\inv: \mathrm{Box}_\zeta \to \mathrm{Box}_\zeta$ be the involution characterized by $g(\inv(v))= g(v)^{-1}$. 
 There is an involution $\inv_R: I\cX_\zeta \to I\cX_\zeta$ such that
$$
\inv_R (\cX_{\zeta,v}) = \cX_{\zeta,\inv(v)}, \quad  w_{\zeta,v}\circ \inv_R = -w_{\zeta,\inv(v)}. 
$$
We have
$$
\ev_\infty: L^\circ_d \to \cX_{\zeta, \inv(v(d)) }
$$
where $v(d) \in \mathrm{Box}_\zeta$ is determined by $d\in \Eff_\zeta$. 
Let 
$\cF_d := (L_d^\circ)^{\bC_{\mathsf{q}}^*}$ which is a connected component of $L_d^{\bC^*_{\mathsf{q}}}$. 
Since $\bC_{\mathsf{q}}^*$ acts trivially on $\cF_d$, it acts linearly on the fibers of any $\bC_{\mathsf{q}}^*$-equivariant vector bundle  on $\cF_d$.  If $V$ is a $\bC^*_{\mathsf{q}}$-equivariant vector bundle over $\cF_d$ then
$$
V = \bigoplus_{k \in \bQ}  V_k = V^f \oplus V^m,
$$
where $V_k$ is the subbundle on which $\bC_{\mathsf{q}}^*$ acts by weight $k$, and $V^m =\displaystyle{ \bigoplus_{k \neq 0} V_k}$  (resp. $V^f=V_0$)  is the moving (resp. fixed) part of $V$ under the $\bC^*_{\mathsf{q}}$-action.  Let 
$$
T^1_d:= T_{L_d}\big|_{\cF_d}, \quad T^2_d:= E_d\big|_{\cF_d}. 
$$
Then $T_d^{1,f}=T_{\cF_d}$ is the tangent bundle of $\cF_d$, and $T_d^{1,m} = N_{\cF_d/L_d}$ is the normal bundle of $\cF_d$ in $L_d$.
We also have $T_d^{2,f}=0$ and $T_d^{2,m} = E_d\big|_{\cF_d}$. 
The virtual tangent bundle of $\cF_d$ is 
$$
T^\vir_{\cF_d} = T^{1,f}_d - T^{2,f}_d = T_{\cF_d} \in K_{\bC_{\mathsf{q}}^*}(\cF_d), 
$$
and the virtual normal bundle of $\cF_d$ is
$$
N^\vir_d = T^{1,m}_d - T^{2,f}_d = N_{\cF_d/L_d} - E_d\big|_{\cF_d} \in K_{\bC_{\mathsf{q}}^*}(\cF_d). 
$$

\subsection{Virtual factorizations and virtual fundamental classes} Given any effective class $d\in \Eff_\zeta$,
$\ev_\infty: L_d^\circ \to \cX_{\zeta, \inv(v(d))}$ restricts to a closed toric embedding 
$\iota_{d\to  \inv(v(d))}: \cF_d = (L_d^\circ)^{\bC_{\mathsf{q}}^*} \to \cX_{\zeta,  \inv(v(d)) }$. In the remainder of this subsection, we fix an effective class $d\in\Eff_\zeta$ and let $v=\inv(v(d))$. 

Let $g(v)\in G$ be the element corresponding to $v\in \mathrm{Box}_\zeta$.  There exists a $\Gamma$-invariant subspace $H\subset V^{g(v)}$ such that
$$
\cF_d = [H   \sslash_\zeta G]  \subset \cX_{\zeta,v} = [V^{g(v)}\sslash_\zeta G].
$$
More explicitly, 
$$
H= \Spec\bC[y_1,\ldots,y_{n_+}] \simeq \bC^{n_+} \subset V^{g(v)} = \Spec\bC[y_1,\ldots, y_{n_+}, p_1,\ldots, p_{n_-}] 
\simeq \bC^{n_++n_-}, 
$$
where $\{ y_1,\ldots, y_{n_+}, p_1,\ldots, p_{n_-}\} \subset \{x_1,\ldots, x_N\}$. $\cF_d$ is the zero locus of the section 
\begin{equation}\label{eqn:section}
\beta_d = (p_1,\ldots, p_{n_-}) \in H^0(\cX_{\zeta,v}, B_d)
\end{equation}
where $B_d$ is a  toric vector bundle over $\cX_{\zeta,v}$ with fiber $\Spec\bC[p_1,\ldots, p_{n_-}] \simeq \bC^{n_-}$. 
One can show that $\iota_{d\to v}^*w_{\zeta,v} =0$, which implies $W|_H=0$. Therefore,  $W_v:= W|_{V^{g(v)}} 
\in \bC[y_1,\ldots, y_{n_+}, p_1,\ldots, p_{n_-}]$ is in the ideal
generated by $p_1,\ldots, p_{n_-}$:
$$
W_v(y,p) =\sum_{i=1}^{n_-} W_i(y,p) p_i.
$$
\begin{equation}\label{eqn:cosection}
\alpha^W_d :=  (-W_1,\ldots, -W_{n_-})  \in H^0(\cX_{\zeta,v}, B_d^\vee)
\end{equation} 
is a section of $B_d^\vee$ (the dual of $B_d$) which can also be viewed as a cosection of $B_d$, i.e., 
$\alpha^W_d: B_d \to \cO_{\cX_{\zeta,v}}$. We have
$$
\langle \alpha^W_d,\beta_d \rangle = -w_{\zeta,v} \in H^0\left(\cX_{\zeta,v}, \cO_{\cX_{\zeta,v} }\right) 
$$
and 
$$
\cZ^w_d := Z(dw_{\zeta,v}) \cap \cF_d = Z(\alpha^W_d)\cap Z(\beta_d). 
$$
In this case, the Favero-Kim  virtual factorization is the following Koszul matrix factorization
\begin{equation}  \label{eq:VMF}
\{\alpha_d^W, \beta_d \} = \Big[
\xymatrix{
\bigoplus_{i} \Lambda^{2i} B_d^\vee \ar@/{ }^{1pc}/[r]^{ \partial = i_{\beta_d} +\alpha_d^W \wedge}       &  \bigoplus_{i} \Lambda^{2i+1} B_d^\vee \ar@/{ }^{1pc}/[l]^\partial
}
\Big]
\end{equation}
where $i_{\beta_d}$ is the interior product with $\beta_d$, and $\alpha^W_d\wedge$ is the wedge product with $\alpha_d$.
\begin{itemize}
\item The cosection $\alpha_d^W$ depends on the superpotential $W_v: V^{g(v)}\to \bC$ but the
section $\beta_d$ does not. The Koszul matrix factorization $\{\alpha_d^W, \beta_d\}$ is
regular in the sense that $\beta_d$ is a regular section of $B_d$. 
\item The cosection $\alpha_d^W$ is not unique, but different choices define quasi-isomorphic
Koszul matrix factorizations. Moreover, $\iota^* B_d = N_{\cF_d/\cX_{\zeta,v}}$ is the normal bundle of the
closed regular embedding $\iota := \iota_{d\to v}: \cF_d\to \cX_{\zeta,v}$, and 
$$
\iota^* \alpha^W_d =  \Big(-\frac{\partial W}{\partial p_1}(y,0),\cdots, -\frac{\partial W}{\partial p_{n_-}}(y,0) \Big)  
\in H^0\big(\cF_d, N^\vee_{\cF_d/\cX_{\zeta,v}}\big)
$$
is independent of choice of $\alpha_d^W$. 
\end{itemize}  

Let
$$
n_d := n_- = \rank_\bC B_d =  \dim_{\bC} \cX_{\zeta,v} -\dim_\bC \cF_d
$$
be the codimension of $\cF_d$ in $\cX_{\zeta,v}$. The Favero-Kim virtual fundamental class is
$$
F_d^w := \tdch^{\cX_{\zeta,v}}_{\cZ^w_d} \{ \alpha_d^W, \beta_d\} \in 
\bH^{2n_d}_{\cZ^w_d}\Big(\cX_{\zeta,v}, (\Omega^\bullet_{\cX_{\zeta,v}}, -dw_{\zeta,v}\wedge)\Big)
$$
which is a class in the hypercohomology group of the twisted Hodge complex $(\Omega^\bullet_{\cX_{\zeta,v}}, -dw_{\zeta,v}\wedge)$.  The precise definition of $\tdch$ (which is quite involved) can be found in \cite{FK}. 
In particular, when $w_{\zeta,v}=0$, $\{ \alpha_d^W, \beta_d\} = \{0, \beta_d\}$ is 
the Koszul complex of the closed regular embedding $\iota:\cF_d\to \cX_{\zeta,v}$, 
$\cZ_d^w =\cF_d$, and
$$
F_d^w = \tdch^{\cX_{\zeta,v}}_{\cF_d} \{ 0, \beta_d\} 
\in \bH^{2n_d}_{\cF_d}\Big(\cX_{\zeta,v}, (\Omega^\bullet_{\cX_{\zeta,v}}, 0)\Big). 
$$ 
We also define
$$
F_d := \iota_*1  =\tau  \in H^{2n_d}_{\cF_d}(\cX_{\zeta,v}) := H^{2n_d}_{\cF^{an}_d}(\cX^{an}_{\zeta,v}, \underline{\bC})
$$
where $\underline{\bC}$ is the constant sheaf on $\cX_{\zeta,v}^{an}$ associated to $\bC$, and   $\tau$ is the Thom class of the closed regular embedding $\iota:\cF_d\to \cX_{\zeta,v}$.  Recall that 
the Gysin map 
$$
\iota_*: H^k(\cF_d) \to H^{k+2n_d}_{\cF_d}(\cX_{\zeta,v}) 
$$
is an isomorphism of $\bC$-linear spaces.  For any $a\in H^k(\cF)$, 
$$
\iota^*\iota_* a = e(N_{\cF_d/\cX_{\zeta,v}})\cup a
$$
where $e(N_{\cF_d/\cX_{\zeta,v}}) = c_{2n_d}(N_{\cF_d/\cX_{\zeta,v} })$ is the Euler class (and the top Chern class) of
the normal bundle $N_{\cF_d/\cX_{\zeta,v}}$.

\subsection{Projection formulas} 
The class $F^w_d$ satisfies the following property. Let
$$
a \in \bH^k_{\cY}\Big(\cX_{\zeta,v},  (\Omega^\bullet_{\cX_{\zeta,v}}, dw_{\zeta,v}\wedge) \Big)
$$
where $\cY$ is a closed substack of $\cX_{\zeta,v}$. Then
$$
\iota^* a  \in  \bH^k_{\cZ^w_d \cap \cY}\Big(\cF_d,  (\Omega^\bullet_{\cF_d}, 0) \Big),\quad
F^w_d \wedge a \in \bH^{2n_d+k}_{\cZ^w_d\cap \cY} \Big(\cX_{\zeta,v},  (\Omega^\bullet_{\cX_{\zeta,v}},0 ) \Big). 
$$
By the proof of \cite[Proposition 4.10]{FK},  if $\cZ^w_d\cap \cY$ is proper over $\bC$ then
\begin{equation}\label{eqn:projection-hyper}
\int_{\cX_{\zeta,v}} F^w_d \wedge a  = \int_{\cF_d} \iota^* a.
\end{equation}
which is zero unless $k=2\dim_{\bC} \cF_d \Leftrightarrow 2n_d+k = 2\dim_{\bC}\cX_{\zeta,v}$.

Similarly, let
$$
a \in H^k_{\cY}(\cX_{\zeta,v})
$$
where $\cY$ is a closed substack of $\cX_{\zeta,v}$. Then
$$
\iota^* a  \in  H^k_{\cZ^w_d \cap \cY} (\cF_d),\quad
F^w_d \wedge a \in H^{2n_d+k}_{\cZ^w_d\cap \cY}(\cX_{\zeta,v}) 
$$
If $\cZ^w_d\cap \cY$ is proper over $\bC$ then
\begin{equation}\label{eqn:projection}
\int_{\cX_{\zeta,v}} F_d \wedge a  = \int_{\cF_d} \iota^* a.
\end{equation}
The left hand side of \eqref{eqn:projection-hyper} (resp. \eqref{eqn:projection}) depends
only on the class of $F_d^w\wedge a$ (resp. $F_d\wedge a$) in $H^{2n_d +k}_c(\cX_{\zeta,v})$. 
The right hand side of \eqref{eqn:projection-hyper} (resp. \eqref{eqn:projection}) depends
only on the class of $\iota^*a$  in $H^k_c(\cF_d)$.

\subsection{$I$-functions}  We first introduce some notation. Let $\kappa:= \dim_{\bC} G$. 
Let $\{ \xi_1^*,\ldots, \xi_\kappa^* \}$ be a $\bZ$-basis  of $\hG\simeq \bZ^\kappa$. Introduce Novikov variables $y=(y_1,\ldots, y_{\kappa})$, 
and define
$$
y^d  := \prod_{a=1}^\kappa y_a^{\langle \xi_a^*, d\rangle}, \quad d\in \Eff_\zeta.
$$
Let $P_a \in \Pic(\cX_\zeta)$ be the line bundle associated to the $G$-character $\xi_a^*$, and let $p_a := - c_1(P_a) \in H^2(\cX_\zeta;\bZ)$. 
Let  $\iota_v: \cX_{\zeta,v}\to \cX_\zeta$ be the inclusion.
We define the GLSM $I$-function of the abelian GLSM $(V, G, \bC_R^*, W, \zeta)$ to be
$$
I_w = z \sum_{v\in \mathrm{Box}_\zeta} I_{w,v}, \quad
I_{w,v}=  e^{ \sum_{a=1}^\kappa (\log y_a)  \iota_v^* p_a/z} \sum_{\substack{ d\in \Eff_\zeta \\ v(d)=v} } y^d \frac{ (\inv_R)_*F_d^w}{e_{\bC_{\mathsf{q}}^*}(\tN^\vir_d)}
$$
where $\tN^\vir_d \in K_{\bC_{\mathsf{q}}^*}(\cX_{\zeta,v})$ is a natural extension of $N^\vir_d \in K_{\bC_{\mathsf{q}}^*}(\cF_d)$, $e_{\bC^*_q}(\tN^\vir_d)$
is the $\bC^*_q$-equivariant Euler class of $\tN^\vir_d$, and  
$$
(\inv_R)_* F_d^w \in \bH^{2n_d}_{\inv_R(\cZ_d^w)}\Big(\cX_{\zeta,v(d)}, (\Omega_{\cX_{\zeta,v(d)}}^\bullet, w_{\zeta,v(d)}\wedge)\Big).
$$ 
In particular, when $W=0$, we obtain the following GLSM $I$-function of $(V, G, \bC_R^*, 0, \zeta)$:
$$
I_0 = z \sum_{v\in \mathrm{Box}_\zeta} I_{0,v}, \quad
I_{0,v}=  e^{ \sum_{a=1}^\kappa (\log y_a)  \iota_v^* p_a/z} \sum_{\substack{ d\in \Eff_\zeta \\ v(d)=v} } y^d \frac{ (\inv_R)_* F_d^0}{e_{\bC_{\mathsf{q}}^*}(\tN^\vir_d)}
$$
where
$$
(\inv_R)_*F_d^0 \in \bH^{2n_d}_{\inv_R(\cF_d)} \Big(\cX_{\zeta,v}, (\Omega_{\cX_{\zeta,v}}^\bullet, 0)\Big).
$$
When $W=0$, we also define the following non-equivariant $I$-function of $(V,G,\bC_R^*,0, \zeta)$:
$$
I = z \sum_{v\in \mathrm{Box}_\zeta} I_v, \quad
I_v =  e^{ \sum_a \log y_a  \iota_v^* p_a/z} \sum_{\substack{ d\in \Eff_\zeta \\ v(d)=v} } y^d \frac{(\inv_R)_*F_d}{e_{\bC_{\mathsf{q}}^*}(\tN^\vir_d)}
$$
where 
$$
F_d = (\iota_{d\to \inv(v(d)) })_*1  \in H^{2n_d}_{\cF_d}(\cX_{\zeta,  \inv(v(d)) }),
\quad  (\inv_R)_*F_d \in H^{2n_d}_{\inv_R(\cF_d)} (\cX_{\zeta,v(d)}). 
$$

\subsection{Matrix factorizations and quasimap central charges} \label{sec:MF} 
Following \cite{PV11}, a {\em matrix factorization} for $(\cX_\zeta,w_\zeta)$ is a pair
$$
(E, \delta_E) = \Big( \xymatrix{ E_0 \ar@<0.6ex>[r]^{\delta_0} & E_1 \ar@<0.6ex>[l]^{\delta_1} } \Big) 
$$
where $E_0$ and $E_1$ are vector bundles (locally free sheaves of $\cO_{\cX_\zeta}$-modules) of finite rank on $\cX_\zeta$,  and 
$$
\delta_1 \circ \delta_0 = w_\zeta \cdot \id_{E_0}, \quad \delta_0\circ \delta_1 = w_\zeta \cdot \id_{E_1}.
$$

Let $\MF(\cX_\zeta, w_\zeta)$ be the dg-category of matrix factorizations of $w_\zeta$ for $\cX_\zeta$. 
There is a well-defined forgetful map 
\begin{equation}\label{eq:MFtoK-zeta} 
\forget: K(\MF(\cX_\zeta,w_\zeta)) \lra K(\cX_\zeta)  = K(\mathrm{Perf}(\cX_\zeta)) 
\end{equation} 
sending a matrix factorization $(E=E_0\oplus E_1, \delta_E)$ to $[E_0]-[E_1]$. 

In \cite{AL23} we consider various versions of category of B-branes. We
define the quasimap central charge of a B-brane in terms of suitable versions of $I$-function, Iritani's Gamma class,  and Chern character of
the K-theory class of the B-brane. In this paper we will  consider the following two versions.  

\begin{definition}[GLSM quasimap central charges of matrix factorizations] \label{central-charge-MF}
Suppose that $\cZ_d^w = \Crit(w_{\zeta, v(d)}) \cap \cF_d$ is proper for all $d\in \Eff_\zeta$. 
Given a matrix factorization $\fB\in \MF(\cX_\zeta, w_\zeta)$,  let $[\fB] \in K(\MF(\cX_\zeta, w_\zeta))$ be
its K-theory class. The GLSM quasimap central charge of $\fB$ is defined to be
\begin{equation}\label{eq:GLSM-central-charge}
Z_w(\fB) = Z_w([\fB]):= (I_w(y, -1), \hat{\Gamma} \ch_w([\fB]))
\end{equation}
which is a $\bC$-valued function in $y$,
where $\hat{\Gamma}$ is a version of Iritani's Gamma class defined in \cite{AL23} and 
$$
\ch_w([\fB]) \in \bigoplus_{v\in \mathrm{Box}_\zeta} \bH^{\even}\left(\cX_{\zeta,v}, (\Omega^*_{\cX_{\zeta,v}}, dw_{\zeta,v}\wedge)\right).
$$
is the Chern character defined in \cite{CKS}.
\end{definition}
 
\begin{remark}
We use the sign convention of~\cite{FK} in the definition of the Chern character  which is the opposite of the sign convention in~\cite{CKS}.
\end{remark}

\begin{definition}[non-equivariant quasimap central charges] \label{central-charge-Perf} 
Suppose that $\cF_d$ is proper for all $d\in \Eff_\zeta$. The non-equivariant quasimap central charge of $F \in K(\cX_\zeta)$ is defined to be
\begin{equation}\label{eq:non-equivariant-central-charge} 
Z(F):= ( I(y,-1), \hat{\Gamma}\ch(F))
\end{equation}
where 
$$
\ch(F) \in \bigoplus_{v\in \mathrm{Box}_\zeta} H^*(\cX_{\zeta,v}) 
$$
is the combinatorial Chern character defined in \cite[Section 5]{BH}. 
\end{definition}

\begin{remark}
When there is a good lift $\tilde{\zeta} \in \widehat{\Gamma}\otimes \bQ$ of $\zeta \in \hG\otimes \bQ$ , one may define 
$J^\epsilon$-functions $J^\epsilon_w$ and $J^\epsilon$    
using $\epsilon$-stable LG quasimap graph spaces,  and 
define $\epsilon$-dependent quasimap central charges $Z^\epsilon_w$ and $Z^\epsilon$ 
using $J^\epsilon_w$ and $J^\epsilon$ (instead of the $I$-functions $I_w$ and $I$), respectively. 
Under the semi-positive assumption, the GLSM  $J$-function $J_w=J_w^\infty$ (resp. the non-equivariant $J$-function $J = J^\infty$) and the GLSM $I$-function 
$I_w=J^{0+}$ (resp. the non-equivariant $I$-function $I=I^{0+}$) should be related by a mirror map ($\epsilon$-wall-crossing). The central charges of
a compact semi-positive toric orbifold \cite{Ir09, Fa20} can be expressed in terms of the $J$-function of the compact orbifold  and is the analogue of $Z^\infty$. 
\end{remark}

\begin{proposition}[comparison of quasimap central charges]
Suppose that $\cF_d$ is proper for any $d\in \Eff_\zeta$.  For any $\fB\in \MF(\cX_\zeta, w_\zeta)$, 
$$
Z_w(\fB) = Z\left(\forget[\fB] \right)
$$
where $[\fB]\in K(\MF(\cX_\zeta,w_\zeta))$ is the $K$-theory class of $\fB$, and $\mathfrak{forget}$ is 
the natural forgetful map \eqref{eq:MFtoK-zeta}. 
\end{proposition}

\subsection{Higgs-Coulomb correspondence} \label{sec:HC}

We say a GLSM $(V,G, \bC^*_\bR, W,\zeta)$ is Calabi-Yau if
$G\subset \SL(V)$,  or equivalently, the GIT quotient   $X_\zeta = V\sslash_\zeta G$ is a Calabi-Yau variety. 
When the GLSM is abelian, the Calabi-Yau condition is equivalent to 
\begin{equation}\label{eq:abelian-CY}
\sum_{i=1}^N D_i  =0 
\end{equation}
where $N=\dim V$. 

In this case we will often call our quasimap central charges simply central charges since they correspond 
to (under the mirror map) central charges of D-branes represented by the matrix factorizations
in the physics literature \cite{Ho00}.

One of the main results of~\cite{AL23} is the so-called Higgs-Coulomb correspondence, which means that 
central charges of an abelian, Calabi-Yau GLSM have explicit Mellin-Barnes type integral representations.

Let $\fB \in \MF([V/G], w)= \MF_G(V, W)$ be a $G$-equivariant matrix factorization of $(V, W)$. There is a natural forgetful map
\begin{equation}\label{eq:MFtoK}
\mathfrak{forget}: K(\MF([V/G],w)) \to K([V/G]) = K_G(V) 
\end{equation} 
sending a $G$-equivariant matrix factorization $(E=E_0\oplus E_1, \delta_E)$ to $[E_0]-[E_1]$. 

Let $\fg$ denote the Lie algebra of the gauge group $G$, and let $\bullet$ denote a point.  We have
$$
K_G(V) \simeq K_G(\bullet) = \bZ[P_1^{\pm 1},\ldots, P_\kappa^{\pm 1}], \quad G = \mathrm{Spec}\bC[P_1^{\pm 1},\ldots, P_\kappa^{\pm 1}]  \simeq (\bC^*)^\kappa;
$$ 
$$
H^*_G(V;\bC)\simeq  H^*_G(\bullet;\bC) =\bC[\sigma_1,\ldots, \sigma_\kappa], \quad \fg =\mathrm{Spec}\bC[\sigma_1,\ldots,\sigma_\kappa] \simeq \bC^\kappa.
$$
Given
$$
\bft =\sum_{a=1}^\kappa t_a \xi_a^* \in \hG, 
$$
where $t_a \in \bZ$,  let
$$
\cL_{\bft} = \prod_{a=1}^\kappa P_a^{t_a} \in \Pic([V/G]= \Pic_G(V) 
$$
be the associated $G$-equivariant line bundle on $V$.

Consider the $(2\pi\sqrt{-1})$-modified equivariant Chern character: \begin{equation}
  \label{eq:eqCh}
  \begin{aligned}
    &\mathrm{Ch} \; : \; K([V/G])  \to H^{*}_{an}([V/G]),  \;\;\;
    \mathrm{Ch} = (2\pi\sqrt{-1})^{\deg/2} \circ \ch \\ &\mathrm{Ch}(\cL_{\bf t}) =
     e^{2\pi\sqrt{-1}\langle \bft, \sigma\rangle} = e^{2\pi\sqrt{-1}\sum_{a=1}^\kappa t_a \sigma_a}
  \end{aligned}
\end{equation}
where $H^{*}_{an}([V/G]) \simeq H^{*}_{G, an}(\bullet)$ is the analytic completion of the equivariant cohomology 
$H^*_G(\bullet)$ of a point $\bullet$,
and $\sigma =\sum_{a=1}^\kappa \sigma_a \xi_a  \in \fg$.

Assuming the Calabi-Yau condition \eqref{eq:abelian-CY}, we define the
integral density of the hemisphere partition function to be
\begin{equation}
  \Gamma(\sigma) \cdot \Ch(\forget([\fB])) = \prod_{i=1}^N  \Gamma(\langle D_i,  \sigma\rangle + \frac{R_i}{2}) \cdot
  \Ch(\forget([\fB]))
  \in H^{*}_{an}([V/G]).
\end{equation}
This expression can be treated as either an equivariant cohomology class or just a complex 
analytic function on $\kappa$ variables $\sigma_1,\ldots, \sigma_\kappa$. The (non-equivariant limit of the)
hemisphere partition function of
the matrix factorization $\fB$ is defined to be~\cite{AL23, HR13} an analytic function of
complex variables $\theta_a = \zeta_a + 2\pi\sqrt{-1}B_a$, $1\leq a\leq  \kappa$: 
\begin{equation}
  \label{eq:hemi}
  Z_{D^2}(\fB) := \frac{1}{2\pi \sqrt{-1}}\oint_{\delta + \sqrt{-1}\fg_\bR} \dd \sigma \,
  \Gamma(\sigma) \cdot e^{ \langle \theta, \sigma\rangle} \Ch\left(\forget([\fB])\right),
\end{equation} 
where  $\fg_\bR=\bigoplus_{a=1}^\kappa \bR \xi_a$, and $\delta \in \fg_\bR$ satisfies the condition 
\begin{equation}\label{eq:separate-poles}
\langle D_i, \delta\rangle + \frac{R_i}{2}>0, \quad i=1,\ldots, N.
\end{equation} 
\begin{lemma}
If there exists $\delta \in \fg_\bR$ which satisfies \eqref{eq:separate-poles} then 
$\cF_d$ is proper for any $d\in \Eff_\zeta$ for any stability condition $\zeta \in \hG$ such that
$V_G^{s}(\zeta) = V_G^{ss}(\zeta)$. 
\end{lemma}

Since $K_{G}(V)$ is generated by line bundles and $Z_{w}(\fB)$ depends only on
$\forget([\fB])$ we can reduce the computation of $Z_{w}(\fB)$ to computations of
$Z(\cL_{\bt})$.
\begin{theorem}[Higgs-Coulomb correspondence~\cite{AL23}]\label{thm:HC} 
The formal power series $Z_{w}(\fB|_{\cX_{\zeta}})$ has a non-zero convergence radius
in the variables $\{y_{a}\}_{a=1}^{\kappa}$ for any $\fB \in \mathrm{MF}([V/G],w)$. Moreover, for any
$\bt \in \hG$ there exist
a non-empty open region $U_{\bt} \subset \fg^{*}_{\bR}$ such that the function 
$$
Z_{D^{2}}(\cL_{\bt}) :=  \frac{1}{2\pi \sqrt{-1}}\oint_{\delta + \sqrt{-1}\fg_\bR} \dd \sigma \,
  \Gamma(\sigma) \cdot e^{ \langle \theta, \sigma\rangle} \Ch(\cL_{\bt})
  $$
is well-defined for $\mathrm{im}(\theta) \in U_{\bt}$.
$$
Z_{D^2}(\cL_{\bt}) = Z(\cL_{\bt}|_{\cX_\zeta})|_{\log y_{a} = -\theta_{a}}
$$
for $\theta$ in a non-empty open region in $\fg^{*}$.
\end{theorem}

Higgs-Coulomb correspondence also allows to look at the Landau-Ginzburg correspondence
from a different perspective. Let $C_{+}, C_{-} \subset \fg^{*}_{\bR}$ be two chambers
of the secondary fan on $\cX_{\zeta}$ adjacent along codimension 1 wall perpendicular to
$h \in \fg_{\bR}$.

Let $\fB \in \MF([V/G], W)$. We say that $\fB$ satisfies the \textit{grade restriction
  rule} with respect to the wall if there exist $B \in \fg^{*}_{\bR}$ such that
\[
  |\langle B+\bt, h \rangle| < \frac14\sum_{i=1}^{N}|\langle D_{i}, h \rangle|
\]
for all characters $\bt \in \hG$ entering $\fB$ nontrivially (underlying $G$-representation $\fB$ contains a representation defined by $\bt$).

\begin{theorem}[Wall-Crossing~\cite{AL23}]
  Let $C_{\pm}, h$ as above and $\zeta_{\pm} \in \mathrm{int}(C_{\pm})$. If
  $\fB$ satisfies grade restriction rule for some $B$, then the functions
  $Z_{w}(\fB|_{\cX_{\zeta_{+}}})|_{\log y_{a} = \theta_{a}}$ and $Z_{w}(\fB|_{\cX_{\zeta_{-}}})|_{\log y_{a} = \theta_{a}}$ are related by
  analytic continuation in $\{\theta_{a}\}$.
\end{theorem}

\section{LG/CY correspondence and wall-crossing in GLSMs} 
\label{sec:LGCY0}

The LG/CY correspondence is a special case of wall-crossing in gauged linear sigma models (GLSMs). 

\subsection{GLSM of the  quintic 3-fold}\label{sec:GLSMquintic} 
Let $V= \Spec\,\bC[x_1,\ldots,x_5,p] \simeq \bC^6$.  The weights of 
the gauge group $G=\bC^*$ and $\bC_R^*$ are listed in the following table. 
\begin{center}
  \begin{tabular}{l|cccccc } 
   & $x_1$ & $x_2$ & $x_3$ & $x_4$ & $x_5$ & $p$ \\ \hline 
    $G$ & $1$ & $1$ & $1$ & $1$ & $1$ & $-5$  \\ \hline
    $\bC^{*}_{R}$ & $0$ & $0$ & $0$ & $0$ & $0$ & $1$
  \end{tabular}
\end{center}
In particular, $G\cap \bC^*_R$ is trivial, and $\bC_\omega = \bC_R^*$. 

Let $x=(x_1,\ldots, x_5)$.  The superpotential is 
 $$
 W(x,p)  = p W_5(x)  =  \sum_{i=1}^5 p x_i^5. 
 $$

 The stability condition $\zeta\in \bZ-\{0\} \subset \bZ = \hG$ is a non-zero integer. 
 \begin{itemize}
 \item In the CY phase $\zeta>0$,  
$$
\cX_\zeta = \left((\bC^5-\{0\})\!\times\!\bC\right)/G = K_{\bP^4}.
$$
The critical locus of $w_\zeta$ is the Fermat quintic Calabi-Yau 3-fold:
$$
\Crit(w_\zeta)= \{ p=W_5(x)=0\} = X_5  \subset \bP^4 = \{p=0\} \subset K_{\bP^4}. 
$$

The GLSM invariants are, up to sign, GW invariants of $X_5$ \cite{CL}. 

\item  In the LG phase $\zeta<0$,   
$$
\cX_\zeta = [ \left(\bC^5 \!\times\! (\bC-\{0\})\right)/\bC^*] = [\bC^5/\mu_5].
$$
The critical locus of $w_\zeta$ is a fat stacky point at the origin, and 
the reduced critical locus is 
$$
\Crit(w_\zeta)_{\mathrm{red}} = [0/\mu_5] = B\mu_5. 
$$
The  GLSM invariants are FJRW invariants of  $(W_5,\mu_5)$.
\end{itemize}

\subsection{$I$-functions  in the CY phase $\zeta>0$}
$$
I\cX_\zeta =\cX_\zeta =K_{\bP^4}, \quad 
\Eff_\zeta = \{d: d\in \bZ_{\geq 0}\}.
$$
$$
V_d =  H^0(\bP^1, \cO_{\bP^1}(d))^{\oplus 5}, \quad W_d = H^1(\bP^1, \cO_{\bP^1}(-5d-1)).
$$
The degree $d$ LG loop space is 
$$
L_d = (V_d- 0)/\bC^* \simeq \bP^{5d+4}.
$$
The obstruction bundle is
$$
 E_d  = ((V_d-0)\times W_d)/\bC^* = \cO_{\bP^{5d+4}} (-5)^{\oplus 5d}. 
$$
For all $d\in \Eff_\zeta =\bZ_{\geq 0}$, 
$$
\cF_d =\bP^4.
$$

We now compute the GLSM $I$-function for the superpotential
$W= p W_5(x)$  and for the zero superpotential $W=0$. 

When $W=pW_5(x)$, the GLSM $I$-function is 
$$
I_w(q,z) = z e^{ (\log q) H/z}  \sum_{d=0}^\infty q^d \frac{\prod_{m=1}^{5d} (-5H-mz)}{\prod_{m=1}^d (H+mz)^5} 
\tdch_{X_5}^{K_{\bP^4}}\{W_5,p\}.
$$
When $W=0$, the GLSM $I$-function is 
$$
I_{w=0}(q,z) = z e^{ (\log q) H/z}  \sum_{d=0}^\infty q^d \frac{\prod_{m=1}^{5d} (-5H-mz)}{\prod_{m=1}^d (H+mz)^5} 
\tdch_{\bP^4}^{K_{\bP^4}}\{0,p\}.
$$

\subsection{$I$-functions in the LG phase $\zeta<0$}
$$
I\cX_\zeta = \bigsqcup_{v=0}^4 \cX_{\zeta,v}, \quad
\cX_{\zeta,v} =\begin{cases}
[\bC^5/\mu_5], & v=0,\\
[0/\mu_5]=B\mu_5, & v\in \{1,2,3,4\}
\end{cases}.
$$
$$
\Eff_\zeta= \{ d: -5d-1\in \bZ_\geq 0\} = \left\{ -\frac{k+1}{5}: k\in \bZ_{\geq 0} \right\}.
$$
$$
V_{-\frac{k+1}{5} } = H^0(\bP^1, \cO_{\bP^1}(k)),\quad 
W_{-\frac{k+1}{5}} = H^1\left(\bP^1, \cO_{\bP^1}(-\frac{k+1}{5} ) \right) ^{\oplus 5}
$$
The LG loop space is 
$$
L_{-\frac{k+1}{5} } =  (V_{-\frac{k+1}{5}} -0)/\bC^* = \bP[5^{k+1}]:=\bP[\underbrace{5,\ldots, 5}_{k+1} ].
$$
The obstruction bundle is
$$
E_{-\frac{k+1}{5}} =  \left( (V_{-\frac{k+1}{5} } -0)\times W_{-\frac{k+1}{5}} \right)/\bC^*
 = \cO_{\bP[5^{k+1}]}(-1)^{\oplus 5 (\lceil \frac{k+1}{5}\rceil-1) }.
$$
For all $k\in \bZ_{\geq 0}$, 
$$
\cF_{-\frac{k+1}{5}} = B\mu_5.
$$

For both $W=pW_5$ and $W=0$, the GLSM $I$-function
\begin{eqnarray*}
I_w(t,z) &=& z \sum_{\substack{ k\in \bZ_{\geq 0}\\ k\notin 4+ 5\bZ} } q^{-(k+1)/5}  
\frac{\Gamma(1-\{\frac{k+1}{5}\} )^5  }{\Gamma(1-\frac{k+1}{5}))^5 } 
\frac{1}{k!} z^{5\lfloor \frac{k+1}{5}\rfloor -k} \one_{ \{ \frac{k+1}{5} \} }  \\
&=&  z \sum_{ \substack{ k\in \bZ_{\geq 0}\\ k\notin 4+ 5\bZ }} q^{-(k+1)/5}  
\frac{\Gamma(1-\{\frac{k+1}{5}\} ))^5}{\Gamma(1-\frac{k+1}{5}))^5 } 
\frac{1}{k!} \frac{ \one_{ \{ \frac{k+1}{5} \} } }{z^{5\{ \frac{k+1}{5}\} -1}   }   \\
&=& z  \sum_{ \substack{ k\in \bZ_{\geq 0}\\ k\notin 4+ 5\bZ }} q^{-(k+1)/5}  
\frac{\Gamma(1-\{\frac{k+1}{5}\} ))^5 }{\Gamma(1-\frac{k+1}{5}))^5} 
\frac{1}{k!} \frac{ \phi_{5 \{\frac{k}{5}\} } }{z^{5\{ \frac{k}{5}\}}   } 
\end{eqnarray*}
where $\phi_k = \one_{\frac{k+1}{5}}$, $k=0,1,2,3$. 

\subsection{Central charges and wall-crossing}
The central charges for the quintic GLSM can be computed using the Higgs-Coulomb
correspondence. Let $B \in \bR$ and $\fB \in \MF([V/G], w)$ such that for any character $\bt \in \hG \simeq \bZ$
entering the underlying representation of $\fB$ the grade restriction rule
\[
  |B + \bt| < 5/2
\]

is satisfied. Then the GLSM central charges $Z_{w}(\fB|_{K_{\bP^{4}}})$ and $Z_{w}(\fB|_{[\bC^{5}/\mu_{5}]})$
are computed by the hemisphere partition function:
\begin{equation} \label{eq:hemisphereQuintic}
  Z_{D^{2}}(\fB) = -\frac{1}{2\pi\sqrt{-1}}\int_{\sqrt{-1}\bR + \delta} \dd \sigma \,
  \Gamma(\sigma)^{5}\Gamma(1-5\sigma)q^{-\sigma}\Ch(\forget[\fB]),
\end{equation}
where $\delta \in (0,1/5)$ and $q = \exp(\zeta+2\pi\sqrt{-1}B)$.
\begin{remark}
  The hemisphere partition function provides the analytic continuation between
  GW theory of the quintic threefold $X_5$ and FJRW theory of the affine LG model $(W_5, \mu_5)$. Such an analytic continuation
  was established in~\cite{CR10} in a different basis. In~\cite{CIR} the analytic
  continuation was found to be consistent with Orlov equivalence. In the GLSM
  language hemisphere partition function provides the analytic continuation between
  $Z_{w}(\fB|_{K_{\bP^{4}}})$ and $Z_{w}(\fB|_{[\bC^{5}/\mu_{5}]})$, so Orlov equivalence is satisfied automatically.
\end{remark}

\section{Open/closed correspondence and extended LG/CY correspondence} 
\label{sec:extended}


The B-model disk potential satisfies the extended Picard-Fuchs equation \eqref{eq:ePF}. Correspondingly, it has the following Mellin-Barnes type integral representation \cite[Section 3]{Wa07}:
\begin{equation}
  \label{eq:MB1}
  \cT(q) = \frac{\pi^{2}}{2}\cdot \frac{1}{2\pi i} \oint_{i\bR} \dd \sigma \, \frac{\Gamma(-\sigma+1/2)\Gamma(5\sigma+1)\Gamma(\sigma+1/2)}
  {\Gamma(\sigma+1)^{5}}e^{i\pi(\sigma-1/2)}q^\sigma,
\end{equation}
where $\cT(q) = \cT^{\mathrm{CY}}(q)$ for the small $q$ expansion. 
 We rewrite this integral using reflection relation for the Gamma function:
\begin{equation}
  \Gamma(x)\Gamma(1-x) = \frac{\pi}{\sin(\pi x)},
\end{equation}
and we also replace $\sigma$ by $-\sigma$:
\begin{equation} \label{eq:MB2}
\begin{aligned}
 &  \cT(q) =\frac{1}{64\pi^3} \times \\
 & \times \frac{1}{2\pi i} \oint_{i \bR+\delta} \dd \sigma \, \Gamma(\sigma)^{5}\Gamma(1-5\sigma)
  \Gamma\big(\sigma+\frac{1}{2}\big)\Gamma\big(-\sigma+\frac{1}{2}\big)
  q^{-\sigma} e^{4\pi i \sigma}(1-e^{-2\pi i \sigma})^5,
\end{aligned} 
\end{equation}
where $0 <\delta < 1/5$.
We observe below that the second line of \eqref{eq:MB2} is nothing but a
quasimap central charge for a particular (abelian, Calabi-Yau) GLSM (see Section \ref{sec:MF}-\ref{sec:HC}).
In this paper we call this particular GLSM the extended GLSM corresponding to the real quintic. 

\subsection{Extended GLSM corresponding to the real quintic} 
More precisely, the input data of the extended GLSM corresponding to the real quintic is
$$
(V = \Spec\, \bC[x_1,\ldots,x_5, p, u,v]\simeq \bC^8 ,\, G=\bC^*\times \mu_2,\, \bC_R^*,\, W,\, \zeta)
$$
where the dimension 8 of $V$ is equal to the number of gamma functions on the right hand side of \eqref{eq:MB2}, 
$\mu_2 =\{\pm 1\}$,  and $G\cap \bC_R^*=\mu_2 \subset \bC_R^*$.  The weights of $G_0=\bC^*$ (the connected component of 1 of the gauge group $G$)  and $\bC^*_\omega = \bC_R^*/\mu_2$ on $V$ can be read from the Gamma function arguments; the weights
of $G_0=\bC^*$, $\mu_2$, $\bC_R^*$, $\bC_\omega^*$ are summarized in the following table: 
\begin{center}
  \begin{tabular}{c|cccccc|cr}
     &   $x_1$ & $x_2$ & $x_3$ & $x_4$ & $x_5$ & $p$ & $u$ & $v$  \\ \hline 
    $G_0=\bC^*$ & $1$ & $1$ & $1$ & $1$ & $1$ & $-5$ & $1$ & $-1$ \\ \hline
    $\mu_2$ & $+$ & $+$ & $+$ & $+$ & $+$ & $+$ & $-$ & $-$ \\ \hline
    $\bC^{*}_\omega$ & $0$ & $0$ & $0$ & $0$ & $0$ & $1$ & $1/2$ & $1/2$\\ \hline
    $\bC^*_R$ &  $0$ & $0$ & $0$ & $0$ & $0$ & $2$ & $1$ & $1$\\ 
  \end{tabular}
\end{center}
where we distinguished the last two columns for later convenience. It is clear, that the above extended GLSM  is obtained from the GLSM of the quintic threefold (see Section \ref{sec:GLSMquintic})  by an addition of two coordinates  $u$ and $v$ of weights $1$ and $-1$ with respect to $G$. The new coordinates have weights $1/2$ with respect to $\bC^{*}_\omega$ and therefore have to be orbifolded by an additional $\mu_{2}$-symmetry (see discussion below on how $G$ is related to $\Gamma$).

Definition of the superpotential $W$ is ambiguous, but we can take a simple family of polynomials such as:
\begin{equation} \label{eq:suPot}
  W_t(x,p,u,v) = p\sum_{i=1}^{5} x_{i}^{5}+tuv, \quad t\in \bC.
\end{equation}

Let $\Gamma$ be the subgroup of $\GL(V)=\GL_8(\bC)$ generated by $G$ and $\bC_R^*$. Then $\Gamma\simeq \bC^* \times \bC_R^*$ acts on 
$V$ by 
$$
 (t_1,t_2)\cdot (x, p, u, v) = ( t_1 x , t_1^{-5} t_2^2 p, t_1 t_2 u, t_1^{-1} t_2 v).
$$
For any $t\in \bC$, 
$$
(t_1,t_2)\cdot W_t(x,p,u,v) = (t_2)^2 W_t(x,p,u,v).
$$ 
So  $\chi:\Gamma = (\bC^*)^2 \to \bC^*_\omega$ is given by $\chi(t_1,t_2)= t_2^2$, and 
$$
G= \mathrm{Ker}(\chi) = \{ (t_1,t_2)\in (\bC^*)^2: (t_2)^2=1\}\simeq \bC^* \times \{\pm 1\}.
$$
The stability condition is a rational $G$ character $\zeta \in \hG \otimes \bQ \simeq \bQ$. 

\subsection{Central charges in the positive phase $\zeta>0$} \label{sec:central-charge+} 
The GIT quotient stack $\cX_\zeta =[V\sslash_\zeta G]$  for the positive choice of the stability parameter $\zeta$ is
$$
 \cX_\zeta =  \big[ \tot\big(\cO_{\bP^{5}}(-5) \oplus \cO_{\bP^{5}}(-1) \big)\big/\mu_2 \big]
 $$
 where $-1$ acts on $\tot(\cO_{\bP^{5}}(-5) \oplus \cO_{\bP^{5}}(-1))$ by 
\begin{equation} \label{eqn:minus-one}
-1 \cdot [x, p, u,v] = [x,p, -u,-v] =  [-x, -p, u,v], \quad x = (x_1,\ldots, x_5).
\end{equation} 
 
\subsubsection*{Case 1: $t\ne 0$.}  For $t \ne 0$ the critical locus of the supertotential~\eqref{eq:suPot} is
\begin{eqnarray*}
\Crit(w_\zeta) &=&  X_5\times B\mu_2 \stackrel{W_5(x) =0}{\subset} \bP^4 \times B\mu_2  \stackrel{u=0}{\subset} [ \bP^5/\mu_2]  \\
&& \stackrel{p=v=0}{ \subset}\big[\tot \big( \cO_{\bP^{5}}(-5) \oplus \cO_{\bP^{5}}(-1) \big)\big/\mu_2\big].
\end{eqnarray*}
Therefore, in this case there are no additional central charges compared to the quintic example. 
On the level of integral representations, for $t\ne 0$ all the branes 
$$
\fB \in \MF([V/G],w) = \MF_G(V, W_t = p W_5(x) + tuv)  
$$
contain the matrix factorization of $uv$. In particular, the Chern character contains a factor $(1+e^{2\pi \sqrt{-1}\sigma})$ that 
cancels all the poles of the two additional gamma functions
$$
\Gamma(\sigma+\frac{1}{2})\Gamma(-\sigma+\frac{1}{2}) = \frac{\pi}{\cos(\pi\sigma)}.
$$

\subsubsection*{Case 2: $t=0$} For $t=0$, the critical locus becomes larger. Let 
$$
CX_5 = \{ W_5(x)=0\} \subset \bP^5
$$
be the cone over the Fermat quintic threefold $X_5\subset \bP^4$, and let
$\cO_{CX_5}(-1)$ be the restriction of $\cO_{\bP^5}(-1)$ to $CX_5$.  For $t=0$ the critical locus of the superpotential is 
singular and is the union of two irreducible components: 
$$
\Crit(w_\zeta) = \{ W_5(x) = px_i^4 =0\} = C_1 \cup C_2
$$
where 
$$
C_1  = \big[\tot \big(\cO_{CX_5}(-1) \big)/\mu_2\big] 
\stackrel{W_5=p=0}{\subset}  \cX_\zeta= \big[\tot \big( \cO_{\bP^{5}}(-5) \oplus \cO_{\bP^{5}}(-1) \big)\big/\mu_2\big].
$$
is a smooth closed substack of codimension 2, and 
$$
C_2 = \{ x_1^4=\cdots= x_5^4=0\} \subset \cX_\zeta
$$
is a non-reduced closed substack of codimension 5. 

The disk potential $\cT^{\CY}(q)$ appears as an additional central charge in this case for a particular choice of a brane 
$$
\fB \in \MF(\cX_\zeta, w_\zeta)  = \MF_G(V^{ss}_G(\zeta), W_0 = p W_5(x)).
$$
Let $\cD_i \subset \cX_\zeta$ be the toric divisor defined by $x_i=0$, and let
$$
U = \bigoplus_{i=1}^5 \cO_{\cX_\zeta}(\cD_i). 
$$
The closed substack $\bigcap_{i=1}^5 \cD_i = (C_2)_{\mathrm{red}}  \subset \cX_\zeta$ 
is the zero locus of the regular section 
$$
b =(x_1,\ldots, x_5) \in H^0(\cX_\zeta, U). 
$$
Let 
$$
a = (px_1^4, \ldots, px_5^4) \in H^0(\cX_\zeta, U^\vee).
$$
Then
$$
\langle a, b\rangle = p\sum_{i=1}^5 x_i^5 = W_0.
$$
The brane $\fB = \cL \otimes \fB_{0}$ is a tensor product of
a line bundle $\cL = \cO_{\cX_{\zeta}}(\cD_{1})^{\otimes 2}$ with the Koszul matrix factorization
\begin{equation}
  \label{eq:KMF}
\fB_0= \{a, b\} = \Big[
\xymatrix{
\bigoplus_i \Lambda^{2i} U^\vee  \ar@/{ }^{1pc}/[r]^{\partial = i_b + a\wedge}  &   \bigoplus_i \Lambda^{2i+1} U^\vee \ar@/{ }^{1pc}/[l]^\partial
}
\Big]
\end{equation}
where $i_b$ is the interior product with the section $b$ of $U$ and  $a\wedge$ is the wedge product with the section $a$ of $U^\vee$.   The matrix factorization \eqref{eq:KMF}
is a deformation of the  Koszul resolution of the closed substack $\bigcap_{i=1}^5 \cD_i \subset \cX_\zeta$.  
The Chern character of $\mathfrak{forget}[\fB]\in K(\cX_\zeta)$ is equal to
\begin{equation}
  \label{eq:chKoszul}
  \Ch(\mathfrak{forget}[\fB]) = e^{4\pi \sqrt{-1} \sigma}(1-e^{-2\pi\sqrt{-1}\sigma})^{5}.
\end{equation}

We have the following open/closed correspondence  relating  
disk invariants $\cT^{\CY}(q)$ of $(X_5, \bR X_5)$ to closed string invariants $Z(\fB)$ of the extended GLSM. 
\begin{proposition}[open/closed correspondence for $(X_5, \bR X_5)$ at $\epsilon = 0^+$] \label{prop:openHC} 
\begin{equation} \label{eq:openHC}
 Z_{D^2}(\fB) = 64 \pi^3 \cT^{\CY}(q) \Big|_{q=e^{-\theta}}.
\end{equation}
\end{proposition}
\begin{proof} 
This follows from \eqref{eq:hemi}, \eqref{eq:MB2}, and \eqref{eq:chKoszul}.
\end{proof}

\subsection{Central charges in the negative phase $\zeta<0$}  \label{sec:central-charge-hybrid}

The brane $\fB$ satisfies the grade restriction rule with respect to the unique wall 
separating the chambers $\zeta > 0$ and $\zeta < 0$.
Indeed, the grade restriction rule takes the form:
\[
  |B+\bt| < 3.
\]
The brane $\fB$ consists of $\bt \in \{-3,-2,-1,0,1,2\}$, so for $B \in (0,1)$ the brane
satisfies the grade restriction rule. Then
$Z_{D_{2}}(\fB)$ realizes the analytic continuation between $Z_{w}(\fB_{\zeta > 0})$
and $Z_{w}(\fB_{\zeta < 0})$. Notice that the analogous brane for the original
quintic GLSM would restrict to zero in the geometric phase and, in particular, would be
not grade restricted.

The wall-crossing of the extended model provides
GLSM viewpoint on the LG/CY correspondence~\eqref{eq:LGCY1} in
the open string sector.

To compute the decomposition of $Z_{D_{2}}(\fB)$ into $\cT^{\LG}$ and
a closed period of the FJRW model we can decompose the class of the brane
$\forget[\fB] \in K([V/G])$. Let $t \in \hG$ be the fundamental character
such that $\forget[\fB] = t^{2}(1-t^{-1})^{5}$.

From the formula~\eqref{eq:MB2} it is clear,
that $\cT^{\LG}$ must come from a part of the brane whose class is proportional to
$1-t^{-5}$ to cancel the poles of $\Gamma(1-5\sigma)$ and the closed period $\cT_{c}$
must come from the part proportional to $1+t^{-1}$ to cancel the poles of
$\Gamma(1/2-\sigma)$. Therefore, to find representation of the form
\[
  \cT^{\CY} = \cT^{\LG} + \cT_c
\]
we need to decompose:
\[
  t^{2}(1-t^{-1})^{5} = f(1+t^{-1})+g(1-t^{-5}).
\]
Such a decomposition is easy to find:
\[
  t^{2}(1-t^{-1})^{5} = (-15t^{2}+10t-10t^{-1}+15t^{-2})(1+t^{-1})+16t^{2}(1-t^{-5}).
\]
Then, using~\eqref{eq:openHC} we have
\[
  \begin{aligned}
  &\cT^{\LG} = \frac{\sqrt{-1}}{2\pi^{2}}\frac{1}{2\pi\sqrt{-1}}\int_{\bR+\delta}
  \dd \sigma \, \frac{\Gamma(\sigma)^{5}}{\Gamma(5\sigma)}\Gamma(1/2-\sigma)
  \Gamma(1/2+\sigma)q^{-\sigma}e^{-\pi\sqrt{-1}\sigma} \\
  &\cT_c = \frac{1}{32\pi^{2}}\frac{1}{2\pi\sqrt{-1}}\int_{\bR+\delta}
  \dd \sigma \, \Gamma(\sigma)^{5}\Gamma(1-5\sigma)q^{-\sigma}e^{-\pi\sqrt{-1}\sigma} \times \\ &\times (-15e^{4\pi\sqrt{-1}\sigma}+10e^{2\pi\sqrt{-1}\sigma}-10e^{-2\pi\sqrt{-1}\sigma}
  +15e^{-4\pi\sqrt{-1}\sigma}).
  \end{aligned}
\]

Notice, that the second line is (up to a factor and coordinate change $q \to -q$) the integral representation of
the quintic hemisphere partition function~\eqref{eq:hemisphereQuintic} for the brane
$\fB_{c}$ with $\Ch(\forget[\fB_{c}]) = -15e^{4\pi\sqrt{-1}\sigma}+10e^{2\pi\sqrt{-1}\sigma}-10e^{-2\pi\sqrt{-1}\sigma}
  +15e^{-4\pi\sqrt{-1}\sigma}$.
  

\subsection{The fan and minimal anticones} \label{sec:fan} 
The coarse moduli space $X_\zeta$ of the GIT quotient stack $\cX_\zeta$ is 
a semi-projective simplicial toric variety defined by a fan. 
Choice of the fan is non-canonical and depends on the choice of coordinates on the torus. 
For the extended model, the fan in the positive and negative phases have the same 1-dimensional cones
whose generators $v_1,\ldots, v_8 \in \bZ^7$ can be chosen to be the column vectors of the following $7\times 8$ matrix: 
\begin{equation}
  v = \begin{pmatrix}
        1 & 1 & 1 & 1 & 1 & 1 & 1 & 1 \\
        1 & 0 & 0 & 0 & -1 & 0 & 0 & 0 \\
        0 & 1 & 0 & 0 & -1 & 0 & 0 & 0 \\
        0 & 0 & 1 & 0 & -1 & 0 & 0 & 0 \\
        0 & 0 & 0 & 1 & -1 & 0 & 0 & 0 \\
        0 & 0 & 0 & 0 & 0 & 0 & 1 & 1 \\
        0 & 0 & 0 & 0 & 2 & 0 & -1 & 1
      \end{pmatrix}
\end{equation}
 The $5 \times 6$ submatrix in the upper-left corner is the standard choice of 1-dimensional cones 
 for the $K_{\bP^{4}}$. We modify this fan by adding 2 vectors and 2 coordinates. Resulting 8 vectors
 satisfy the relation $\sum_{i=1}^{5} v_{i} -5v_{6}+v_{7}-v_{8} = 0$.
 In addition, the coefficients are chosen to represent the correct $\mu_{2}$ action.

The set of minimal anticones is
$$
\cA^{\min}_{\zeta} =\begin{cases}
\big\{ \{1\}, \{2\}, \{3\}, \{4\}, \{5\}, \{7\} \big\}, & \zeta>0; \\
\big\{ \{6\}, \{8\} \big\}, & \zeta<0. 
\end{cases}
$$
Minimal anti-cones are in one-to-one correspondence with torus fixed points in $X_\zeta$. 

\subsection{$I$-functions in the positive phase $\zeta>0$} 
\begin{equation}\label{eq:b013}
\mathrm{Box}_\zeta = \Big\{ 
b_0= 0, \ 
b_1 = \frac{1}{2}(v_7+v_8) = \begin{pmatrix}1\\ 0\\ 0\\ 0\\ 0\\ 1\\ 0\end{pmatrix}, \ 
b_3 = \frac{1}{2}\sum_{i=1}^6 v_i =\begin{pmatrix} 3\\ 0\\ 0\\ 0\\ 0\\ 0\\ 1 \end{pmatrix} 
\Big\} 
\end{equation}
where $\age(b_k)= k$. The three elements $b_0, b_1, b_3\in \mathrm{Box}_\zeta$ correspond to  the following three elements in $G\subset (\bC^*)^8$:
\begin{equation}\label{eq:gb013}
\begin{aligned}
g(b_0)  &=  (1,1,1,1,1,1,1,1) \quad \text{(identity)} \\
g(b_1) &=  (1,1,1,1,1,1,-1,-1), \\
g(b_3) &=  (-1, -1, -1, -1, -1, -1,1,1).
\end{aligned} 
\end{equation}
The inertia stack of $\cX_\zeta$ is a disjoint union of three connected components: 
$$
I \cX_\zeta = \cX_{\zeta,0} \cup \cX_{\zeta, 1} \cup \cX_{\zeta,3} 
$$
where  $\cX_{\zeta,0} = \cX_{\zeta}$ is the untwisted sector, and
\begin{eqnarray*}
\cX_{\zeta,1} &:= &\cX_{\zeta,b_1} =   \{ u=v=0\}\simeq  K_{\bP^4} \times B\mu_2,\\
\cX_{\zeta, 3} &:= & \cX_{\zeta, b_3} =   \{ x_i=p=0\}\simeq B\mu_2\times \bC.
\end{eqnarray*}

As a graded complex vector space, the orbifold Chen-Ruan cohomology of $\cX_\zeta$ (which is the state space of
orbifold quasimap theory to the GIT quotient stack $\cX_\zeta = [V\sslash_\zeta G]$) is  
$$
H^*_{\CR}(\cX_\zeta) =  H^*(\cX_{\zeta, 0})\one_0  \oplus H^*(\cX_{\zeta,1})\one_1 \oplus  H^*(\cX_{\zeta,3}) \one_3,
$$
where  $\deg \one_k = 2k$, and 
$$
H^*(\cX_{\zeta,0}) \simeq H^*(\bP^5),\quad  H^*(\cX_{\zeta,1})\simeq H^*(\bP^4), \quad H^*(\cX_{\zeta,3}) = H^*(\bullet),
$$
Therefore, the Poincar\'{e} polynomial is 
\begin{equation}\label{eq:CR-plus}
P_t\left(H^*_{\CR}(\cX_\zeta)\right) := \sum_k \dim_{\bC} t^k H^k_{\CR}(\cX_\zeta) =  1 + 2t^2 + 2t^4 + 3t^6 + 2t^8 + 2t^{10}.
\end{equation} 

As a graded complex vector space, the A-model state space of the extended GLSM  in the positive phase is 
$$
\cH_W = \cH_{W,0}\one_0 \oplus \cH_{W,1}\one_1 \oplus \cH_{W,3}\one_3 
$$
where $\deg \one_k = 2k-4$, and
$$
\cH_{W,k}=   \bH^*\left(\cX_{\zeta, k},  (\Omega^*_{\cX_{\zeta,k}}, dw_{\zeta,k}\wedge) \right).
$$

Let  $\cZ^w_{\zeta,k} = \Crit(w_{\zeta,k}) \subset \cX_{\zeta,k}$. We consider the following cases:
\begin{center}
 \renewcommand{\arraystretch}{1.2}
 \begin{tabular}{l|c|c|c}
 & $\cZ_{\zeta,0}^w$ & $\cZ_{\zeta,1}^w$ &  $\cZ_{\zeta,3}^w$ \\ \hline 
(1) $W=0$ & $\cX_{\zeta,0}=\cX_\zeta$  & $\cX_{\zeta,1} = K_{\bP^4}\times B\mu_2$  & $\cX_{\zeta,3} = B\mu_2\times \bC$ \\ \hline
(2) $W=pW_5(x)$ &  $C_1\cup C_2$  & $X_5\times  B\mu_2$ & $B\mu_2\times \bC$ \\  \hline 
(3) $W=uv$ & $K_{\bP^4}\times B\mu_2$ & $K_{\bP^4}\times B\mu_2 $ &  $\emptyset$  \\  \hline
(4) $W=pW_5(x)+ uv$ & $X_5\times B\mu_2$ & $X_5\times B\mu_2$ & $\emptyset$ 
\end{tabular}
\renewcommand{\arraystretch}{1}
\end{center}
where $\emptyset$ denotes the empty set.  In cases (1)--(3), $\cZ_{\zeta,0}^w$ is not proper and $\cH_{W,0}$ is infinite dimensional.
In case (4), or more generally, when $W= pW_5(x) + t uv$ where $t\neq 0$, 
$$
\cH_W \simeq H^*(X_5) \oplus H^*(X_5),
$$ 
so Poincar\'{e} polynomial is 
\begin{equation}\label{eq:GLSM-plus}
P_t(\cH_W) :=  \sum_k t^k \dim_{\bC} \cH_W^k  = 2P_t(H^*(X_5)) =  2 + 2t^2 + 408 t^3+ 2t^4 + 2t^6. 
\end{equation} 

The effective classes are $\Eff_\zeta = \{ d \in \frac{1}{2}\bZ \mid d \geq 0 \}$. Let
\begin{eqnarray*}
V_d &=& H^0(\bP^1, \cO_{\bP^1}(d))^{\oplus 5} \oplus H^0(\bP^1, \cO_{\bP^1}(d-\frac{1}{2}))\\
W_d &=&  H^1(\bP^1, \cO_{\bP^1}(-5d-1)\oplus H^1(\bP^1, \cO_{\bP^1}(-d-\frac{1}{2}))
\end{eqnarray*}
Note that $V_0=\bC^5$ and $W_0=0$, so 
$$
L_0 = E_0=\bP^4\times B\mu_2.
$$
The degree $d$ LG loop space is 
$$
L_d = [(V_d- 0)/G] =  \begin{cases} 
\left[\bP^{6d+4} /\mu_2 \right] \text{ where -1 acts on the last $d$ coordinates},   & d\in \bZ_{>0}; \\ 
\left[\bP^{6d+1}/\mu_2 \right]  \text{ where -1 acts on the last $d-\frac{1}{2}$ coordinates}, & d\in \frac{1}{2}+\bZ_{\geq 0}.
\end{cases} 
$$
The degree $d$ obstruction bundle is 
$$
E_d  = [(V_d-0)\times W_d)/\bC^* = \begin{cases}
\left[  \cO_{\bP^{6d+4}} (-5)^{\oplus 5d}\oplus \cO_{\bP^{6d+4}}(-1)^{\oplus d } /\mu_2\right] , & d\in \bZ_{>0};\\
\left[ \cO_{\bP^{6d+1}} (-5)^{\oplus (5d+1/2)}\oplus   \cO_{\bP^{6d+1}} (-1)^{\oplus (d-1/2) }/\mu_2 \right], & d\in \frac{1}{2} + \bZ_{\geq 0}.
  \end{cases}
$$
$$
\dim L_d -\rank E_d = \begin{cases}
4, & d\in \bZ_{\geq 0};\\
1, & d\in \frac{1}{2}+\bZ_{\geq 0}
\end{cases}.
$$
We have
$$ 
\cF_d= \begin{cases}
\bP^4 \times B \mu_2 \stackrel{p=0}{\subset} \cX_{\zeta,1} = \tot(\cO_{\bP^4}(-5))\times B\mu_2,  & d\in \bZ_{>0}; \\
B\mu_2  \stackrel{v=0}{\subset}  \cX_{\zeta,3} = B\mu_2 \times \bC, & d\in \frac{1}{2} +  \bZ_{\geq 0}.
\end{cases}
$$
For any superpotential $W$, 
$$
z^{-1} I_w = I_{w,0} + I_{w,1} + I_{w,3}  
$$
where $I_{w,0}=0$. In cases (3) and (4), $I_{w,3}=0$. 

\begin{enumerate}
\item If $W=0$ then
\begin{eqnarray*}
I_{w,1} (q,z) &=& e^{\log q H/z} \cdot \sum_{d=0}^\infty q^d \frac{\prod_{m=1}^{5d} (-5H-mz) \prod_{m=1}^{d}(-H-(m-\frac{1}{2})z)} {\prod_{m=1}^d (H+mz)^5  \prod_{m=1}^d (H+(m-1/2)z)} \one_1 \cdot \tdch^{K_{\bP^4}\times B\mu_2}_{\bP^4\times B\mu_2} \{0,p\}  \\
&=& e^{\log q H/z}  \cdot \sum_{d=0}^\infty q^d \frac{\prod_{m=1}^{5d} (5H+mz) } {\prod_{m=1}^d (H+mz)^5}  \one_1 \cdot
\tdch^{K_{\bP^4}\times B\mu_2}_{\bP^4\times B\mu_2} \{0,p\},
\end{eqnarray*}
\begin{eqnarray*} 
I_{w,3}(q,z) &=&  \sum_{k=0}^\infty q^{k+\frac{1}{2}} 
 \frac{\prod_{m=0}^{5k+2} (-(m+\frac{1}{2})z)\prod_{m=1}^k(-mz)}{\prod_{m=0}^k ((m+\frac{1}{2})z)^5 \prod_{m=1}^k(mz)}\one_3
  \cdot \tdch^{B\mu_2\times \bC}_{B\mu_2}\{0,v\} \\
&= & - \sum_{k=0}^\infty q^{k+\frac{1}{2}}  \frac{\prod_{m=0}^{5k+2} ((m+\frac{1}{2})z)}{\prod_{m=0}^k ((m+\frac{1}{2})z)^5 } \one_3 
\cdot \tdch^{B\mu_2\times \bC}_{B\mu_2}\{0,v\} \\
&=& -4 \sum_{\substack{ d>0 \\ d\text{ odd} }} q^d \frac{(5d)!!}{ (d!!)^5} \frac{\one_3}{z^2}  \tdch^{B\mu_2\times \bC}_{B\mu_2}\{0,v\}.
\end{eqnarray*} 

\item If $W=pW_5(x)$ then
\begin{eqnarray*}
I_{w,1} (q,z) &=& e^{\log q H/z}  \cdot \sum_{d=0}^\infty q^d \frac{\prod_{m=1}^{5d} (5H+mz) } {\prod_{m=1}^d (H+mz)^5}  \one_1 \cdot
\tdch^{K_{\bP^4}\times B\mu_2}_{X_5\times B\mu_2} \{ W_5(x),p\} ,\\
I_{w,3}(q,z) &=&  -4 \sum_{\substack{ d>0 \\ d\text{ odd} }} q^d \frac{(5d)!!}{ (d!!)^5} \frac{\one_3}{z^2}  \tdch^{B\mu_2\times \bC}_{B\mu_2}\{0,v\}.
\end{eqnarray*}

Let $\fB$ be defined as in Section \ref{sec:central-charge+}. Then
\begin{eqnarray*}
Z_w(\fB) &=& \langle I_w(q,-1), \hat{\Gamma}\ch_w([\fB])\rangle\\
&=& 8\sum_{\substack{ d>0 \\ d\text{ odd} }} q^d \frac{(5d)!!}{ (d!!)^5}\Gamma(\frac{1}{2})^6 2^5 \int_{B\mu_2}1 \\
&=&  128 \pi^3 \sum_{\substack{ d>0 \\ d\text{ odd} }} q^d \frac{(5d)!!}{ (d!!)^5} \\
&=&  64\pi^3\cT^{\CY}(q). 
\end{eqnarray*} 
By  \eqref{eq:openHC},
$$
Z_{D^2}(\fB) = 64\pi^3\cT^{CY}(q)\Big|_{q=e^{-\theta}}.
$$
This is consistent with Higgs-Coulomb correspondence (Theorem \ref{thm:HC}).

\item If $W= uv$ then 
\begin{eqnarray*}
I_{w,1} (q,z) &=& e^{\log q H/z}  \cdot \sum_{d=0}^\infty q^d \frac{\prod_{m=1}^{5d} (5H+mz) } {\prod_{m=1}^d (H+mz)^5}  \one_1 \cdot
\tdch^{K_{\bP^4}\times B\mu_2}_{\bP^4\times B\mu_2} \{ 0,p\}, \\
I_{w,3}(q,z) &=& -4 \sum_{\substack{ d>0 \\ d\text{ odd} }} q^d \frac{(5d)!!}{ (d!!)^5} \frac{\one_3}{z^2}  
\tdch^{B\mu_2\times \bC}_{\emptyset}\{u,v\} = 0.
\end{eqnarray*} 

\item If $W=pW_5(x)+uv$  (or more generally $W=pW_5(x) + tuv$ where $t\neq 0$) then 
\begin{eqnarray*}
I_{w,1} (q,z) &=& e^{\log q H/z}  \cdot \sum_{d=0}^\infty q^d \frac{\prod_{m=1}^{5d} (5H+mz) } {\prod_{m=1}^d (H+mz)^5}  \one_1 \cdot
\tdch^{K_{\bP^4}\times B\mu_2}_{X_5\times B\mu_2} \{ W_5(x),p\}, \\
I_{w,3}(q,z) &=&  0.
\end{eqnarray*} 

\end{enumerate}

\subsection{A mirror conjecture in the positive phase.}
We now specialize to the case (2): $W=pW_5(x)$. Define  
\begin{eqnarray*}
\be &:=&  \tdch^{K_{\bP^4}\times B\mu_2}_{X_5\times B\mu_2} \{W_5(x),p\} \one_1,   \\
\bp &:=& \tdch^{B\mu_2\times \bC}_{B\mu_2}\{0,v\} \one_3. 
\end{eqnarray*} 
Then
$$
(\bp, \one_3) = \int_{B\mu_2\times \bC} \tdch^{B\mu_2\times \bC}_{B\mu_2}\{0,v\} = \int_{B\mu_2} 1 = \frac{1}{2}. 
$$
Then the GLSM $I$-function can be rewritten as
$$
I^+_w(q,z) := I_w(q,z) = z \left( \sum_{k=0}^3 I_k^{\CY}(q) \frac{H^k}{z^k} \be  -2 \cT^{\CY}(q) \frac{\bp}{z^2}\right)
$$

Let $J^+_w(Q,z)$ be the  (yet-to-be-defined)  GLSM (small) $J$-function of the extended GLSM in the positive phase $\zeta>0$. 
\begin{conjecture}[mirror conjecture for the extended GLSM in the positive phase] \label{GLSM-J+}
\begin{equation}\label{eq:GLSM-J+} 
J^+_w(Q,z) = \frac{I^+_w(q,z)}{I^{\CY}_0(q,z)}  \quad \text{under the mirror map}\quad \log Q = \frac{I_1^{\CY}(q )}{I_0^{\CY}(q)}. 
\end{equation}
\end{conjecture}

Define a generating function of genus-zero primary GLSM invariants:
\begin{equation}\label{eq:one3+}
F^+(Q) :=  z (J_w^+(Q,z),\one_3) =\sum_{\substack{d\in\bZ_{>0} \\ d \text{ odd} } } \langle \one_3\rangle_{\frac{d}{2}} Q^{\frac{d}{2}}
\end{equation}
Conjecture  \ref{GLSM-J+} implies
\begin{equation} \label{eq:one3-T+} 
F^+(Q) = -\frac{\cT^{\CY}(q)}{I_0^{\CY}(q)}   \quad \text{under the mirror map}\quad \log Q = \frac{I_1^{\CY}(q )}{I_0^{\CY}(q)}. 
\end{equation} 
\eqref{eq:one3+}, \eqref{eq:one3-T+} and Theorem \ref{CYopen} (open mirror theorem in the Calabi-Yau phase) imply 
$$
F^+(Q)=  -F_{0,1}^{\CY}(Q) 
$$
which is equivalent to the following. 
\begin{corollary}[numerical open/closed correspondence in CY phase]\label{OC+} 
For all positive odd number $d$,
$$
\langle \one_3\rangle_{\frac{d}{2}} =  -N^{\text{disk}}_d.
$$
where $-N^{\text{disk}}_d$ is the disk invariants defined by the non-trivial spin structure on 
the real quintic $\bR X_5 \simeq\bR \bP^3$. 
\end{corollary}
It will be interesting to have a direct A-model proof of the above numerical open/closed correspondence, as in 
\cite{LY22, LY}. 

The GLSM $J$-function $J_w^+$ should be defined in terms of genus-zero $\infty$-stable LG quasimaps to the extended GLSM.
The definition of  $\epsilon$-stability in \cite{FJR18, FK} requires a good lift of $\zeta$, which does not exist for the extended GLSM.
We leave  the precise definition of $J_w^+$, as well as the proofs of Conjecture \ref{GLSM-J+} and Corollary \ref{OC+}, to future work.

\subsection{$I$-functions in the negative phase $\zeta<0$} 
In the negative phase $\zeta<0$, 
$$
 \cX_{\zeta} =  \big[  \tot\big(\cO_{\bP[5,1]}(-1)^{\oplus 6}\big) \big/\mu_2  \big]. 
$$ 

$$
\mathrm{Box}_\zeta =\{ b_0, b_1, b_3\} \cup \{ b_{j,k}: j\in \{1,2,3,4\}, k\in \{0,1\} \} 
$$
where $b_0, b_1, b_3$ are given in Equation \eqref{eq:b013}, 
$$
b_{j,0} = \frac{j}{5} (\sum_{i=1}^5 v_i +v_7) +  (1-\frac{j}{5}) v_8 =
\begin{pmatrix}
j+1\\ 0\\ 0\\ 0\\ 0\\ 1\\ 1
\end{pmatrix}
$$
and
$$
b_{j,1}= \frac{j}{5}\sum_{i=1}^5 v_i + \{ \frac{1}{2}+\frac{j}{5}\} v_7 + \{ \frac{1}{2}-\frac{j}{5}\} v_8
=  \begin{pmatrix}
j+1\\ 0\\ 0\\ 0\\ 0\\ 1\\ 1 +\epsilon_j
\end{pmatrix}, \quad
\epsilon_j =\begin{cases} -1, & j=1,2\\ 1, & j= 3, 4.\end{cases}
$$
We have $\age(b_{j,k})= j+1$ for $j\in\{1,2,3,4\}$ and $k\in \{0,1\}$. 

For $k\in \{0,1,3\}$, $g(b_k) \in G$ is given in Equation \eqref{eq:gb013}. 
Let $\zeta_5 = e^{2\pi\sqrt{-1}/5}$. For $j\in \{1,2,3,4\}$ and $k\in \{0,1\}$, 
\begin{equation}\label{eq:gbjk}
g(b_{j,k})= \left(\xi_5^j, \xi_5^j, \xi_5^j, \xi_5^j, \xi_5^j, 1,  (-1)^k \xi_5^j, (-1)^k \xi_5^{-j}\right) \in G. 
\end{equation} 
The inertia stack of $\cX_\zeta$  is a disjoint union of eleven connected components: 
$$
I\cX_{\zeta} = \cX_0\cup \cX_1 \cup \cX_3 \cup \bigcup_{\substack{ j\in \{1,2,3,4\} \\ k\in \{0,1\}} } \cX_{j,k} 
$$
where 
\begin{eqnarray*}
\cX_0&=& \cX_{\zeta} =\left[ \tot\big(\cO_{\bP[5,1]}(-1)^{\oplus 6}\big)/\mu_2\right] ,\\
\cX_1&:=&  \cX_{\zeta, b_1}= \{ u=v=0\} = [\bC^5/\mu_5] \times B\mu_2, \\
\cX_3 &:=& \cX_{\zeta, b_3}= \{x_i=p=0\} = B\mu_2\times \bC, \\
\cX_{j,k} &:=& \cX_{\zeta, b_{j,k}} = \{ x_i = u=v=0\} = B(\mu_5\times \mu_2).
\end{eqnarray*} 

As a graded complex vector space, the Chen-Ruan orbifold cohomology of $\cX_\zeta$ is 
$$
H^*_{\CR} (\cX_\zeta) = \bigoplus_{k\in \{0,1,3\}}H^*(\cX_k)\one_k \oplus   \bigoplus_{\substack{ j=\{1,2,3,4\} \\ k\in \{0,1\} }} H^*(\cX_{j,k})\one_{j,k} 
$$ 
where $\deg \one_k = 2k$, $\deg \one_{j,k}= 2j+2$, and
$$
H^*(\cX_0) \simeq H^*(\bP[5,1]) =\bC 1 \oplus \bC p, \quad
H^*(\cX_1) \simeq H^*(\cX_3)\simeq H^*(\cX_{j,k})\simeq H^*(\bullet).
$$
The Poincar\'{e} polynomial is
\begin{equation}\label{eq:CR-minus} 
P_t(H^*_{\CR}(\cX_\zeta)) := \sum_k t^k \dim_\bC H^k_{\CR}(\cX_\zeta) = 
1 + 2t^2 + 2t^4 + 3t^6 + 2 t^8 + 2t^{10}
\end{equation}
which agrees with \eqref{eq:CR-plus} in the positive phase, as expected. 

As a graded complex vector space, the A-model state space of the extended GLSM in the negative phase is 
$$
\cH_W =  \bigoplus_{k\in \{0,1,3\} } \cH_{W,k}\one_k \oplus  \bigoplus_{\substack{ j=\{1,2,3,4\} \\ k\in \{0,1\} }} \cH_{W,j,k} \one_{j,k} 
$$
where $\deg \one_k = 2k-4$, $\deg\one_{j,k} = 2j-2$, and 
$$
\cH_{W,k} =  \bH^*\left(\cX_k,  (\Omega^*_{\cX_k}, dw_k\wedge) \right),\quad
\cH_{W,j,k}=   \bH^*\left(\cX_{j,k},  (\Omega^*_{\cX_{j,k}}, dw_{j,k}\wedge) \right).
$$

Let   $\cZ^w_k = \Crit(w_k) \subset \cX_k$ and let $\cZ^w_{j,k}= \Crit(w_{j,k})\subset \cX_{j,k}$. 
Let 
$$
X_0 = \{ W_5(x)=0\} \subset \bC^5,\quad \mathbf{o}=\{x_i^4=0\}\subset \bC^5, 
$$
$$
C^-_1 = \{ W_5(x)= p=0\} \simeq [X_0/\mu_5] \times \bC\times B\mu_5 \subset \cX_\zeta, \quad C^-_2 = \{ x_i^4=0\} \subset \cX_\zeta, 
$$

 \begin{center}
 \renewcommand{\arraystretch}{1.2}
 \begin{tabular}{l|c|c|c|c}
 & $\cZ_0^w$ & $\cZ_1^w$  & $\cZ_3^w$ &  $\cZ_{j, k}^w$  \\ \hline 
(1) $W=0$ & $\cX_\zeta$  & $[\bC^5/\mu_5] \times B\mu_2$ & $B\mu_2\times \bC$  & $B(\mu_5\times \mu_2)$  \\ \hline
(2) $W=pW_5(x)$ &  $C^-_1\cup C^-_2$  & $[\mathbf{o}/\mu_5] \times  B\mu_2$  & $B\mu_2\times \bC$ & $B(\mu_5\times \mu_2)$   \\  \hline 
(3) $W=uv$ & $[\bC^5/\mu_2] \times B\mu_2$ & $[\bC^5/\mu_5] \times B\mu_2 $  &  $\emptyset$  &   $B(\mu_5\times \mu_2)$  \\  \hline
(4) $W=pW_5(x)+ uv$ & $[\mathbf{o}/\mu_5]\times B\mu_2$ & $[\mathbf{o}/\mu_5] \times B\mu_2$ &  $\emptyset$ 
 &  $B(\mu_5\times \mu_2)$ \end{tabular}
\renewcommand{\arraystretch}{1}
\end{center}
In cases (1)--(3), $\cZ_0^w$ is not proper and $\cH_{W,0}$ is infinite dimensional.
In case (4), or more generally, when $W= pX_5(x) + t uv$ where $t\neq 0$,
$$
\cH_{W,0} = \cH_{W,0}^7 =\bC^{204},\;\;
\cH_{W,1} = \cH_{W,1}^5 =\bC^{204},\;\;
\cH_{W,3}=0, \;\;
\cH_{W,j, k}  \simeq H^*(\bullet), 
$$
so the  Poincar\'{e} polynomial is 
\begin{equation}\label{eq:GLSM-minus}
P_t(\cH_W) :=  \sum_k t^k \dim_{\bC} \cH_W^k  = 2 + 2t^2 + 408 t^3+ 2t^4 + 2t^6 
\end{equation}
which agrees with \eqref{eq:GLSM-plus} in the positive phase, as expected. 

The effective classes are
\begin{eqnarray*}
\Eff_\zeta &=& \{ d: -5d-1\in \bZ_{\geq 0} \} \cup \{ d: -d-\frac{1}{2} \in \bZ_{\geq 0}\} \\
&=& \{ -\frac{k+1}{5}: k\in \bZ_{\geq 0}\} \cup \{ -\frac{1}{2}-k: k\in \bZ_{\geq 0}\}. 
\end{eqnarray*}

Let $k\in \bZ_{\geq 0}$. 
\begin{eqnarray*}
V_{-\frac{k+1}{5}} &=& H^0(\bP^1, \cO_{\bP^1}(k)) \oplus H^0\Big(\bP^1, \cO_{\bP^1}(\frac{2k-3}{10})\Big) \\
W_{-\frac{k+1}{5}} &=& H^1\Big(\bP^1, \cO_{\bP^1}(-\frac{k+1}{5}) \Big)^{\oplus 5} \oplus H^1(\bP^1, \frac{-2k-7}{10}) 
\end{eqnarray*}
The degree $-(k+1)/5$ LG loop space is 
$$
L_{-\frac{k+1}{5}} = \left[ \bP[5^{k+1}, 1^{\lfloor \frac{2k-3}{10} \rfloor +1}] \big/\mu_2\right] 
$$
where $-1$ acts on the last  $\lfloor \frac{2k-3}{10} \rfloor +1$ coordinates. The obstruction bundle is
$$
E_{-\frac{k+1}{5}} =\left[ \tot\big(\cO_{\bP}(-1)^{\oplus 5 (\lceil \frac{k+1}{5}\rceil -1) +\lceil \frac{2k+7}{10}\rceil -1}\big) \big/\mu_2\right]
$$
where $\bP = \bP\big[5^{k+1}, 1^{\lfloor \frac{2k-3}{10} \rfloor +1} \big]$. 
$$
\cF_{-\frac{k+1}{5}} = B(\mu_5\times \mu_2)  \subset 
\begin{cases}
\cX_1= [\bC^5/\mu_5]\times B\mu_5, & k\in 5\bZ+4,\\
\cX_{5\{\frac{k+1}{5}\}, 1} =  B(\mu_5\times \mu_2), & k\notin 5\bZ+4.
\end{cases}
$$
 
\begin{eqnarray*}
V_{-\frac{1}{2}-k} &=&  H^0(\bP^1, \cO_{\bP^1}(5k+\frac{3}{2})) \oplus H^0(\bP^1, \cO_{\bP^1}(k)\Big)\\
W_{-\frac{1}{2}-k} &=& H^1\Big(\bP^1, \cO_{\bP^1}(-\frac{1}{2}-k)\Big)^{\oplus 5}  \oplus H^1(\bP^1, \cO(-1-k))
\end{eqnarray*}
The degree $-1/2-k$ LG loop space is 
$$
L_{-\frac{1}{2}-k} = \left[ \bP[5^{5k+2}, 1^{k+1}]\big/\mu_2\right] 
$$
where  $-1$ acts on the last $k+1$ coordinates.  The obstruction bundle is
$$
E_{-\frac{1}{2}-k} =\left[ \tot\big(\cO_{\bP}(-1)^{\oplus 5 (\lceil \frac{1}{2} + k \rceil -1) +k}\big) \big/\mu_2\right]
$$
where $\bP = \bP\big[5^{5k+2}, 1^{k+1} \big]$. 

$$
\cF_{-\frac{1}{2}-k} =B\mu_2 \stackrel{u=0}{\subset} \cX_3 = B\mu_2\times \bC. 
$$

For any superpotential $W$,  
$$
z^{-1} I_w = I_{w,0} +  I_{w,1} +  I_{w,3} +   \sum_{j=1}^4 I_{w,j,0} +  \sum_{j=1}^4 I_{w,j,1}  
$$
where
$$
I_{w,0} = I_{w,1} = I_{w,j,0}=0,
$$ 
and 
\begin{eqnarray*}
\sum_{j=1}^4 I_{w,j,1} &=& \sum_{\substack{ k\in \bZ_{\geq 0} \\ k\notin 4+5\bZ}} q^{-(k+1)/5}
\frac{\Gamma(1- \{ \frac{k+1}{5} \})^5\Gamma(1-\{ \frac{2k+7}{10} \} )\Gamma(1-\{ -\frac{2k-3}{10}\})}{\Gamma(1-\frac{k+1}{5})^5 \Gamma(1-\frac{2k+7}{10})
\Gamma(1+ \frac{2k-3}{10 })}
\frac{1}{k!}\frac{ \one_{5\{ \frac{k+1}{5} \}, 1 } }{ z^{5\{\frac{k+1}{5}\} -1} } \\
&=& \sum_{\substack{ k\in \bZ_{\geq 0} \\ k\notin 4+5\bZ}} q^{-(k+1)/5}
\frac{\Gamma(1- \{ \frac{k+1}{5} \})^5}{\Gamma(1-\frac{k+1}{5})^5}
\frac{(-1)^{\lfloor\frac{2k+7}{10}\rfloor}}{k!}\frac{ \one_{5\{ \frac{k+1}{5} \}, 1 } }{ z^{5\{\frac{k+1}{5}\} -1} }.
\end{eqnarray*}

\begin{enumerate}
\item If $W=0$ then
\begin{eqnarray*}
I_{w,3}&=& \sum_{k=0}^\infty q^{-\frac{1}{2}-k} \frac{ \prod_{m=0}^{k-1} (-(m+\frac{1}{2})z)^5 \prod_{m=1}^k (-mz)}{\prod_{m=0}^{5k+1}((m+\frac{1}{2}) z) \prod_{m=1}^k (mz)} \one_3 \cdot \tdch^{B\mu_2\times \bC}_{B\mu_2}\{ 0,u\}  \\
&=& 4 \sum_{k=0}^\infty q^{-\frac{1}{2}-k} \frac{ ((2k+1)!!)^5 }{(2k+1)^5 (10k+3)!!} 
  \frac{\one_3}{z^2} \cdot \tdch^{B\mu_2\times \bC}_{B\mu_2} \{0,u\} \\
&=& 4 \sum_{d>0, d\text{ odd} } q^{-\frac{d}{2} }\frac{(d!!)^5}{d^5(5d-2)!!}  
 \frac{\one_3}{z^2} \cdot \tdch^{B\mu_2\times \bC}_{B\mu_2} \{0,u\}.
\end{eqnarray*}
\item If $W=pW_5(x)$ then
$$
I_{w,3} = 4 \sum_{d>0, d\text{ odd} } q^{-\frac{d}{2} } \frac{(d!!)^5}{d^5(5d-2)!!}  
\frac{\one_3}{z^2}  \tdch^{B\mu_2\times \bC}_{B\mu_2} \{0,u\}. 
$$

Let $\fB\in \MF(\cX_\zeta, w_\zeta)$ such that $\mathfrak{forget}([\fB]) = 16 t^2(1-t^{-5})$, where $t=\cO_{\cX_\zeta}(\cD_1)$ (see Section \ref{sec:central-charge-hybrid}). Then
\begin{eqnarray*}
Z_w(\fB) &=& \langle I_w(q,-1), \hat{\Gamma}\ch_w([\fB])\rangle\\
&=&  -8 \sum_{\substack{ d>0 \\ d\text{ odd} }} q^{-\frac{d}{2}} \frac{(d!!)^5}{d^5 (5d-2)!!}\Gamma(\frac{1}{2})^6 32 \int_{B\mu_2}1 \\
&=&  -128 \pi^3 \sum_{\substack{ d>0 \\ d\text{ odd} }}q^{-\frac{d}{2}} \frac{(d!!)^5}{d^5 (5d-2)!!}\\
&=& 64 \pi^3 \cT^{\LG}(q). 
\end{eqnarray*}

\item  If $W=uv$ then 
$$
I_{w,3} = 4 \sum_{d>0, d\text{ odd} } q^{-\frac{d}{2} } \frac{(d!!)^5}{d^5(5d-2)!!}   \frac{\one_3}{z^2} \tdch^{B\mu_2\times \bC}_{\emptyset} 
\{ v,u\} =0. 
$$
\item If $W=pW_5(x) + uv$, or more generally $W= pW_5(x) + tuv$ where $t\neq 0$, then
$$
I_{w,3}=0.
$$
\end{enumerate}

\subsection{A mirror conjecture in the negative phase.}
We now specialize to the case (2): $W=pW_5(x)$. Define  
$$
\tphi_k := \cdot \one_{ 5\{\frac{k+1}{5}\}, 1}, \quad k=0,1,2, 3, 
$$
and
$$
\bp_-  :=  \tdch^{B\mu_2\times \bC}_{B\mu_2}\{0,u\} \one_3. 
$$
Then
$$
(\bp_-, \one_3) =\int_{B\mu_2\times \bC} \tdch^{B\mu_2\times \bC}_{B\mu_2}\{0,u\} = \int_{B\mu_2} 1 = \frac{1}{2}. 
$$
Then GLSM $I$-function can be rewritten as (recall that $q=t^{-5}$) 
$$
I^-_w(q,z) := I_w(q,z) = z \left( \sum_{k=0}^3 \epsilon_k I_k^{\LG}(t) \frac{\tphi_k}{z^k}    -2 \cT^{\LG}(t) \frac{\bp_-}{z^2}\right)
$$ 
where $\epsilon_0=\epsilon_1=1$ and $\epsilon_2=\epsilon_3=-1$.

Let $J^-_w(\tau,z)$ be the (yet-to-be-defined) GLSM (small) $J$-function of the extended GLSM in the negative phase $\zeta < 0$. 
\begin{conjecture}[mirror conjecture for the extended GLSM in the negative phase] \label{GLSM-J-}
\begin{equation}\label{eq:GLSM-J-} 
J^-_w(\tau,z) = \frac{I^-_w(t,z)}{I^{\LG}_0(t,z)}  \quad \text{under the mirror map}\quad \tau = \frac{I_1^{\LG}(t)}{I_0^{\LG}(t)}. 
\end{equation}
\end{conjecture}
We leave  the precise definition of $J_w^-$ and the proof of Conjecture \ref{GLSM-J-} to future work.  

Define a generating function of genus-zero primary GLSM invariants:
\begin{equation}\label{eq:one3-}
F^-(\tau) :=  z (J_w^-(\tau,z),\one_3).  
\end{equation}
Conjecture  \ref{GLSM-J-} implies
\begin{equation} \label{eq:one3-T-} 
F^-(\tau) =  -\frac{\cT^{\LG}(t)}{I_0^{\LG}(t)}  \quad \text{under the mirror map}\quad \tau = \frac{I_1^{\LG}(t )}{I_0^{\LG}(t)}. 
\end{equation} 
\eqref{eq:one3-}, \eqref{eq:one3-T-} and Conjecture \ref{LGopen} (open mirror conjecture for $(W_5,\mu_5)$) imply 
the following open/closed correspondence: 
$$
F^-(\tau)=  -F_{0,1}^{\LG}(\tau).
$$

\section{B-model} \label{sec:B-model}
\label{sec:B}
Mirror theorem~\cite{Gi96, LLY97} for the quintic threefold in particular implies that central charges of the quintic are equal to the period integrals of the mirror quintic.  More precisely,
consider the following family $\cQ$ of quintic threefolds in $\bP^4$ over $\bC_t$:
\begin{equation}
	W_t(x)=\sum_{i=1}^5 x_i^5 + t \prod_{i=1}^5 x_i=0.
\end{equation}
This family is invariant with respect to $(\mu_5)^3 \subset \bP^4$ that preserves the product
$\prod_i x_i$. Mirror quintic is the family $[\cQ/\mu_5^3]$. 
The mirror family admits a holomorphic volume form $\Omega$.  Period integrals of $\Omega$ over cycles $\gamma \in H_3(\cQ, \bZ)$ satisfy the Picard-Fuchs equation and are equal to the central charges 
of the quintic for a particular choice of branes.
\begin{equation}
	\omega_{\gamma}(t) = \int_{\gamma} \Omega = \frac{1}{(2\pi\sqrt{-1})^{2}}\int_{T(\gamma)}\frac{\dd^5 x}{W_t(x)},
\end{equation}
where $T(\gamma) \in H_5(\bC^5 \backslash \{W_t = 0\})$ is the Hopf fibration lift of
the tubular neighborhood of $\gamma \in H_3(\cQ_t, \bZ)$. Here we lift a 3-cycle $\gamma$ from a
given fiber to a Gauss-Manin flat family of 3-cycles in $\cQ$.

Another well-known representation for the periods of the quintic is the LG representation by
oscillatory integrals:
\begin{equation}
	\int_{\Gamma} e^{-W_t(x)} \dd^5 x, \;\;\; \Gamma \in H_5(\bC^5, \Re(W) \gg 0).
\end{equation}

\subsection{Open B-model}
There exist a B-model representation for the disk potential $\tau$ \cite{MW09}.
Let
\begin{equation}
	C_{\pm} := \{x_1+x_2=0, \; x_3+x_4=0, \; x_5^2 \pm \sqrt{t}x_1 x_3\} \subset \cQ
\end{equation}
be two families of curves in $\cQ$. They descend to the mirror family $[\cQ/\mu_5^3]$.
Now, consider a family $C$ of 3-chains $C_t \in C_3(\cQ_t, \bZ)$ such that $\partial C_t = [C_{+,t}] - [C_{-,t}]$.
It was shown that the partial periods:
\begin{equation} \label{eq:partialPeriods}
	\int_{C} \Omega
\end{equation}
satisfy the extended Picard-Fuchs equation~\eqref{eq:ePF}.

Here we show that the open-closed correspondence of this paper holds true on the B-side as well,
namely the partial periods are equal to actual (oscillatory) periods associated to the
extended GLSM of Section \ref{sec:extended}.

\subsection{Extended B-model}
In this and the next two subsections, we show how to reproduce the disk potentials in CY~\eqref{eq:Iosc6} and LG~\eqref{eq:LGB1} phases
using mirror symmetry for the extended GLSM.\\

We use a version of Givental-Hori-Vafa mirror construction for $\cX_\zeta =[\bC^8\sslash_\zeta G]$ where
$G=\bC^*\times \mu_2$.  Recall from  Section \ref{sec:fan} that the generators 
of 1-dimensional cones are the column vectors $v_1,\ldots, v_8$ of the matrix
\begin{equation}
  v = \begin{pmatrix}
        1 & 1 & 1 & 1 & 1 & 1 & 1 & 1 \\
        1 & 0 & 0 & 0 & -1 & 0 & 0 & 0 \\
        0 & 1 & 0 & 0 & -1 & 0 & 0 & 0 \\
        0 & 0 & 1 & 0 & -1 & 0 & 0 & 0 \\
        0 & 0 & 0 & 1 & -1 & 0 & 0 & 0 \\
        0 & 0 & 0 & 0 & 0 & 0 & 1 & 1 \\
        0 & 0 & 0 & 0 & 2 & 0 & -1 & 1
      \end{pmatrix}
\end{equation}

The mirror superpotential is given by the formula 
$$
F(y,u,v) = \sum_{i} q_{i}y^{v_{i}} =  y_{0}f(y,u,v)
$$
where 
\begin{equation}
	f(y,u,v) = \sum_{i \le 4} y_i + \frac{v^{2}}{\prod_{i}y_{i}} + t + u/v + uv.
\end{equation} 
The superpotential gets deformed according to the R-charges:
$$ 
F_{q} = F+\sum_{i=1}^8 \frac{R_i}{2}\log(y^{v_i}) 
$$
or $\log(y_{0})+1/2\log(y_{0}uv)+1/2\log(y_{0}u/v) =
2\log(y_0)+\log(u)$.
Thus, the period integrals can be written as
\begin{equation}
  \label{eq:Iosc}
	I_{\Gamma}=\int_{\Gamma} e^{-y_{0}f} y_{0} u \, \dd y_{0} \,\omega,
\end{equation}
where
\begin{equation}
  \label{eq:volumeForm}
  \omega = \frac{\dd u}{u} \, \frac{\dd v}{v} \prod_{i=1}^{4} \frac{\dd y_{i}}{y_{i}}
\end{equation}
is the canonical volume form on the torus $(\bC^{*})^{6}_{u,v,y}$ and
$\Gamma$ is any contour that makes the integral convergent. The integral depends
only on the class of $\Gamma$ in $H_{7}((\bC^{*})^{7}, \Re(F_{q}) \gg 0)$.

\subsection{Extended period in the positive phase $\zeta>0$} \label{sec:period-positive} 

In order to evaluate the integral we perform the following (multivalued) coordinate change:
\begin{equation}
  \label{eq:coord1}
  \begin{aligned}
  &  x_{i} := y_{0}y_{i}, \;\;\; 1 \le i \le 4, \;\;\; x_{5} :=  \frac{ v^{2}y_{0} } { \prod_i y_i} , \\
  &  p_{1} := y_{0}t, \;\;\; p_{2} := y_{0}uv.
  \end{aligned}
\end{equation}
Then 
\begin{equation}
  \label{eq:coord12}
  \begin{aligned}
    & F = \sum_{i=1}^{5}x_{i}+p_{1}+p_{2}+t^{-5}p_{1}^{5}p_{2}/\prod_{i=1}^{5}x_{i},    \\ 
    & y_{0}^{2}u = \prod_{i=1}^{5}x_{i}^{-1/2}p_{1}^{7/2}p_{2}t^{-7/2}, \\
    & \frac{\dd y_{0}}{y_{0}}\omega = \frac12 \prod_{i=1}^{5}\frac{\dd x_{i}}{x_{i}}
      \prod_{a=1}^{2}\frac{\dd p_{a}}{p_{a}}.
  \end{aligned}
\end{equation}

Therefore,
\begin{equation}
  \label{eq:Iosc4}
  I_{\Gamma} = \frac{t^{-1}}2\sum_{m \ge 0}(-1)^{m}\frac{t^{-5(m+1/2)}}{m!}\int_{\Gamma}e^{-\sum_{i} x_{i} -\sum_{a}p_{a}} p_{1}^{5(m+1/2)}p_{2}^{m}\prod_{i}x_{i}^{-m-3/2}\prod_{i}\dd x_{i}\prod_{a} \dd p_{a} .
\end{equation}

\paragraph{\bf Integration contour}
We can choose a particular integration contour that transforms the integral in
the product of 7 Gamma function integrals.
We define the integration contour $\Gamma_{geom}$ parametrically, using the Hankel contour $\cC$.
The Hankel contour $\cC \subset \bC$ is a contour that
goes from infinity to 0 along the real axis, goes around the origin on the
counter-clockwise direction and then continues along the positive real axis
to positive infinity, $\log\cC =
(+\infty,0] \cup [0, 2\pi i] \cup [2\pi i, +\infty+2\pi i)$.

To define $\Gamma_{geom}$ we let $y_{0} \in (0, \infty), \; y_{1}, \ldots, y_{4} \in y_{0}^{-1}\cC, \; v^{2} \in \prod_{i=1}^{4} y_{i}/y_{0}
\cC, \; u \in y_{0}^{-1}v^{-1}(0,+\infty)$.

The contour separates the variables in the integral reduces to a product of 7 integrals:
\begin{multline}
  \label{eq:separation}
  \int_{\Gamma}e^{-\sum_{i} x_{i} -\sum_{a}p_{a}} p_{1}^{5(m+1/2)}p_{2}^{m}\prod_{i}x_{i}^{-m-3/2}\prod_{i}\dd x_{i}\prod_{a} \dd p_{a} =\\=
  \prod_{i=1}^{5}\int_{\cC} e^{-x_{i}}x_{i}^{-m-3/2}\dd x_{i} \cdot \int_{0}^{\infty} e^{-p_{1}}
  p_{1}^{5(m+1/2)} \dd p_{1} \cdot \int_{0}^{\infty}e^{-p_{2}}p_{2}^{m} \dd p_{2} = \\
  = 32 \Gamma(-m-1/2)^{5}\Gamma(5m+7/2)m!.
\end{multline}
Plugging it back to the sum we get
\begin{multline}
  \label{eq:Iosc5}
    I_{\Gamma} = t^{-1}16\sum_{m \ge}\sum_{m \ge 0}(-1)^{m}t^{-5(m+1/2)} \Gamma(5m+7/2)\Gamma(-m-1/2)^{5} = \\ = t^{-1}16\Gamma(7/2)\Gamma(-1/2)^{5}t^{-5/2}+ \cdots = -32\pi^{3}t^{-1}(30t^{-5/2}+\cdots)
\end{multline}
Then, after the coordinate change $t \to q = t^{-5}$ we get:
\begin{equation}
  \label{eq:Iosc6}
  t I_{\Gamma} =  -32\pi^{3}\cT^{\CY}(q).
\end{equation}

\subsection{Extended mirror period in the negative phase $\zeta<0$} \label{sec:period-negative}
In the negative phase we expand the integral~\eqref{eq:Iosc} differently.
Here we perform the coordinate change:
\begin{equation}
  \label{eq:coord2}
  \begin{aligned}
    x_{i} := y_{0}y_{i}, \;\;\; 1 \le i \le 4, \;\;\; x_{5} := \frac{ v^{2}y_{0} }{ \prod_{i}y_{i} } , \\
    p_{1} := y_{0}t, \;\;\; p_{2} := y_{0}u/v.
  \end{aligned}
\end{equation}
So that  
\begin{equation}
  \label{eq:coord21}
  \begin{aligned}
    & F = \sum_{i=1}^{5}x_{i}+p_{1}+p_{2}+t^{5}p_{1}^{-5}p_{2} \prod_{i=1}^{5}x_{i}, \\
    & y_{0}^{2}u = \prod_{i=1}^{5}x_{i}^{1/2}p_{1}^{-3/2}p_{2}t^{3/2}, \\
    & \frac{\dd y_{0}}{y_{0}}\omega =  \frac12\prod_{i=1}^{5}\frac{\dd x_{i}}{x_{i}}
      \prod_{a=1}^{2}\frac{\dd p_{a}}{p_{a}}.
  \end{aligned}
\end{equation}
Expanding the exponential in the integral we compute
\begin{equation}
  \label{eq:IoscLG1}
  I_{\Gamma_{LG}} = \frac{t^{-1}}{2} \sum_{m \ge 0} \frac{(-1)^{m}}{m!}t^{5(m+1/2)}\int_{\Gamma_{LG}} e^{-\sum_{i} x_{i}
  -p_{1}-p_{2}}\prod_{i}x_{i}^{m-1/2}p_{1}^{-5(m+1/2)}p_{2}^{m}\dd p_{1}\dd p_{2}\dd^{5}x.
\end{equation}
\paragraph{\bf Contour}

We construct the integration contour $\Gamma_{LG} \simeq \cC^{6} \times \bR$ parametrically in
a similar fashion to the geometric positive case. We let
$y_{0} \in \cC, \; y_{i} \in y_{0}^{-1}\cC$ for $1 \le i \le 4$. Then $v^{2} \in
y_{0}^{-1}\prod_{i=1}^{4}y_{i}\cC$ and $u \in vy_{0}^{-1}(0,\infty)$

The integral factorizes into a product of 7 gamma functions:
\begin{multline}
  \int_{\Gamma_{LG}} e^{-\sum_{i} x_{i}
  -p_{1}-p_{2}}\prod_{i}x_{i}^{m-1/2}p_{1}^{-5(m+1/2)}p_{2}^{m}\dd p_{1}\dd p_{2}\dd^{5}x = \\ =
  64\Gamma(m+1/2)^{5}\Gamma(-5m-3/2)m!.
\end{multline}
Plugging it back into the expression for $I_{\Gamma_{LG}}$ we obtain:
\begin{multline}
  I_{\Gamma_{LG}} = 32t^{-1} \sum_{m \ge 0} t^{5(m+1/2)}\Gamma(-5m-3/2)\Gamma(m+1/2)^{5} = \\ = -64t^{-1}\pi^{3}(-2/3t^{5/2}+\cdots).
\end{multline}
We compute that up to a multiple and the coordinate change $q = t^{-5}$ this expression coincides with
$\cT^{\LG}(q)$:
\begin{equation}
  \label{eq:LGB1}
  t I_{\Gamma_{LG}}|_{t^{-5}=q} = -64\pi^{3} \cT^{\LG}(q).
\end{equation}


\begin{thebibliography}{CFGKS}

\bibitem[AL23]{AL23} K. Aleshkin and C.-C. M. Liu, 
``Higgs-Coulomb correspondence and wall-crossing in abelian GLSMs,"
arXiv:2301.01266. 

\bibitem[BCS]{BCS} L. A. Borisov, L. Chen, and G. Smith, 
``The orbifold Chow ring of toric Deligne-Mumford stacks,"
J. Amer. Math. Soc. {\bf 18} (2005), no. 1, 193--215.

\bibitem[BH]{BH} L. A. Borisov and R. P. Horja, 
``On the K-theory of smooth DM stacks," 
Snowbird lectures on string geometry, 21--42, Contemp. Math., {\bf 401}, Amer. Math. Soc., Providence, RI, 2006.

\bibitem[BCT19]{BCT19} A. Buryak, E. Clader, and R. J. Tessler, 
 ``Closed extended  $r$-spin theory and the Gelfand-Dickey wave function," 
J. Geom. Phys. {\bf 137} (2019), 132--153.

\bibitem[BCT22]{BCT22} A. Buryak, E. Clader, and R. J. Tessler, 
``Open  $r$-spin theory I: Foundations,"
Int. Math. Res. Not. IMRN (2022), no. 14, 10458--10532.

\bibitem[BCT24]{BCT24} A. Buryak, E. Clader, and R. J. Tessler, 
``Open $r$-spin theory II: The analogue of Witten's conjecture for $r$-spin disks,"
J. Differential Geom. {\bf 128} (2024), no. 1, 1--75.

\bibitem[CdGP]{CdGP}  P. Candelas, X. de la Ossa, P. Green, L. Parks,
``A pair of Calabi-Yau manifolds as an exactly soluble superconformal field theory,'' 
Nuclear Physics B {\bf 359} (1): 21--74, 1991.

\bibitem[CKL]{CKL} H.-L. Chang, Y.-H. Kiem, and  J. Li, 
``Algebraic virtual cycles for quantum singularity theories," 
Comm. Anal. Geom. {\bf 29} (2021), no.8, 1749--1774.

\bibitem[CL]{CL} H.-L. Chang and J. Li,
``Gromov-Witten invariants of stable maps with fields," 
Int. Math. Res. Not. IMRN 2012, no. 18, 4163--4217.

\bibitem[CLL]{CLL} H.-L. Chang, J. Li, and W.-P. Li, 
``Witten's top Chern class via cosection localization,"
Invent. Math. {\bf 200} (2015), no.3, 1015--1063.

\bibitem[CCK]{CCK} D. Cheong, I. Ciocan-Fontanine, and B. Kim,
``Orbifold quasimap theory," 
Math. Ann. {\bf 363} (2015), no. 3-4, 777--816. 

\bibitem[CIR]{CIR} A. Chiodo, H. Iritani, and Y. Ruan, 
``Landau-Ginzburg/Calabi-Yau correspondence, global mirror symmetry and Orlov equivalence,"
Publ. Math. Inst. Hautes \'{E}tudes Sci. {\bf 19} (2014), 127--216.

\bibitem[CR10]{CR10}  A. Chiodo and Y. Ruan,
``Landau-Ginzburg/Calabi-Yau correspondence for quintic three-folds via symplectic transformations,"
Invent. Math. {\bf 182} (2010), no. 1, 117--165. 

\bibitem[Ch08]{Ch08} C.-H. Cho, 
``Counting real  $J$-holomorphic discs and spheres in dimension four and six,"
J. Korean Math. Soc. {\bf 45}  (2008), no. 5, 1427--1442.


\bibitem[CKS]{CKS} D. Choa, B. Kim, and B. Sreedhar,
``Riemann-Roch for stacky matrix factorizations,"
Forum Math. Sigma {\bf 10} (2022), Paper No. e108, 29 pp.


\bibitem[CK14]{CK14} I. Ciocan-Fontanine and B. Kim,
``Wall-crossing in genus zero quasimap theory and mirror maps,"
Algebr. Geom., {\bf 1}  (2014), no. 4, 400--448.

\bibitem[CK20]{CK20}   I. Ciocan-Fontanine and B. Kim,
``Quasimap wall-crossings and mirror symmetry," 
Publ. Math. Inst. Hautes \'{E}tudes Sci. {\bf 131} (2020), 201--260.

\bibitem[CKM]{CKM} I. Ciocan-Fontanine, B. Kim, and D. Maulik. 
``Stable quasimaps to GIT quotients,"
J. Geom. Phys. {\bf 75} (2014), 17--47. 

\bibitem[CFGKS]{CFGKS} I. Ciocan-Fontanine, D. Favero, J. Gu\'{e}r\'{e}, B. Kim, and M. Shoemaker,
``Fundamental Factorization of a GLSM, Part I: Construction,"
Mem. Amer. Math. Soc. {\bf 289} (2023), no. 1435, iv+96 pp.

\bibitem[Fa20]{Fa20}  B. Fang, 
``Central charges of T-dual braned for toric varieties,"
Trans. Amer. Math. Soc. {\bf 373} (2020), no. 6, 3829--3851.

\bibitem[FL13]{FL13} B. Fang and C.-C. M. Liu,
``Open Gromov-Witten invariants of toric Calabi-Yau 3-folds,"
Comm. Math. Phys. {\bf 323} (2013), no. 1, 285--328.

\bibitem[FLT22]{FLT22} B. Fang, C.-C. M. Liu, and H.-H. Tseng,
``Open-closed Gromov-Witten invariants of 3-dimensional Calabi-Yau smooth toric DM stacks," 
Forum Math. Sigma {\bf 10} (2022), Paper No. e58, 56 pp.

\bibitem[FK]{FK} D. Favero and B. Kim,
``General GLSM Invariants and Their Cohomological Field Theory,"
arXiv:2006.12182.

\bibitem[FJR07]{FJR07} H. Fan, T. J. Jarvis, and Y. Ruan,
``The Witten equation and its virtual fundamental cycle,"
arXiv:0712.4025.

\bibitem[FJR18]{FJR18}  H. Fan, T. J. Jarvis, and Y. Ruan,
``A mathematical theory of the gauged linear sigma model," 
Geom. Topol. {\bf 22} (2018), no. 1, 235--303 (published version) and arXiv:1506.02109v5 (2020). 

\bibitem[Gi96]{Gi96} A. Givental, ``Equivariant Gromov-Witten invariants,''
Internat. Math. Res. Notices {\bf 1996}, no. 13, 613--663.

\bibitem[GKT]{GKT} M. Gross, T. L. Kelly, and R. J. Tessler, 
``Open FJRW Theory and Mirror Symmetry,"
arXiv:2203.02435. 

\bibitem[Ho00]{Ho00} S. Hosono
``Local Mirror Symmetry and Type IIA Monodromy of Calabi-Yau Manifolds,"
Adv. Theor. Math. Phys. {\bf 4} (2000), 335--376.

\bibitem[HR13]{HR13} K. Hori and M. Romo, 
``Exact Results In Two-Dimensional (2,2) Supersymmetric Gauge Theories With Boundary,” 
arXiv:1308.2438.

\bibitem[Ir09]{Ir09} H. Iritani,
``An integral structure in quantum cohomology and mirror symmetry for toric orbifolds,"
Adv. Math. {\bf 222} (2009), no. 3, 1016--1079.

\bibitem[KL13]{KL13} Y.-H. Kiem and J. Li, 
``Localizing virtual cycles by cosections," 
J. Amer. Math. Soc. {\bf 26} (2013), no. 4, 1025--1050.

\bibitem[KL20]{KL20} Y.-H. Kiem and J. Li, 
``Quantum singularity theory via cosection localization,"
J. Reine Angew. Math. {\bf 766} (2020), 73--107.



\bibitem[LM]{LM} W. Lerche, P. Mayr, 
``On N = 1 mirror symmetry for open type II strings," arXiv:hep-th/0111113.

\bibitem[LLY97]{LLY97} B. Lian, K. Liu, S.-T. Yau, 
``Mirror principle I,"
Asian J. Math. {\bf 1} (1997), no. 4, 729--763.

\bibitem[LY22]{LY22} C.-C. M. Liu and S. Yu, 
``Open/closed correspondence via relative/local correspondence,"
Adv. Math. {\bf 410} (2022), Paper No. 108696, 43 pp.

\bibitem[LY]{LY} C.-C. M. Liu and S. Yu, 
``Orbifold open/closed correspondence and mirror symmetry,"
arXiv:2210.11721.

\bibitem[Ma]{Ma} P. Mayr, 
``N=1 mirror symmetry and open/closed string duality,"
arXiv:hep-th/0108229.

\bibitem[MW09]{MW09} D. R. Morrison and J. Walcher,
``D-branes and normal functions,"  
Adv. Theor. Math. Phys. {\bf 13} (2009), no. 2, 553--598. 

\bibitem[Ni]{Ni} S. Nill, ``Extended FJRW theory of the quintic threefold in genus zero,"
talk in S\'{e}minaire de G\'{e}om\'{e}trie Enum\'{e}rative, 2 December 2021. 

\bibitem[PSW08]{PSW08} R. Pandharipande, J. Solomon, and J. Walcher, 
``Disk enumeration on the quintic 3-fold," 
J. Amer. Math. Soc. {\bf 21} (2008), no. 4, 1169--1209. 

\bibitem[PV11]{PV11} A. Polishchuk and A. Vaintrob, 
``Matrix factorizations and singularity categories for stacks," 
Ann. Inst. Fourier (Grenoble) {\bf 61} (2011), no. 7, 2609--2642.

\bibitem[PV16]{PV16} A. Polishchuk and A. Vaintrob, 
``Matrix factorizations and cohomological field theories," 
J. Reine Angew. Math. {\bf 714} (2016), 1--122.

\bibitem[RR17]{RR17} D. Ross and Y. Ruan, 
``Wall-crossing in genus zero Landau-Ginzburg theory,"
J. Reine Angew. Math. {\bf 733} (2017), 183--201. 

\bibitem[So]{So} J. P. Solomon, 
``Intersection theory on the moduli space of holomorphic curves with Lagrangian boundary conditions,"
MIT thesis, arXiv:math.SG/0606429.

\bibitem[TX18]{TX18} G. Tian and G. Xu, 
``Analysis of gauged Witten equation,"
J. Reine Angew. Math. {\bf 740} (2018), 187–274.

\bibitem[TX21]{TX21}  G. Tian and G. Xu, 
``Virtual cycles of gauged Witten equation,"
J. Reine Angew. Math. {\bf 771} (2021), 1--64.

\bibitem[TX16]{TX16} G. Tian and G. Xu, 
``Correlation functions of gauged linear  $\sigma$-model,"
Sci. China Math. {\bf 59} (2016), no. 5, 823--838.

\bibitem[TX20]{TX20} G. Tian and G. Xu, 
``A wall-crossing formula and the invariance of GLSM correlation functions,"
Peking Math. J. {\bf 3} (2020), no. 2, 235--291.

\bibitem[TX17]{TX17} G. Tian and G. Xu, 
``The symplectic approach of gauged linear  $\sigma$-model,"
G\"{o}kova Geometry/Topology Conference (GGT), G\"{o}kova, 2017, 86–111.

\bibitem[TX]{TX} G. Tian and G. Xu, 
``Gauged Linear Sigma Model in Geometric Phases. I,"
arXiv:1809.00424, 
``Gauged linear sigma model in geometric phases. II. the virtual cycle,"
arXiv:2407.14545.


\bibitem[TZ1]{TZ1}  R. J. Tessler and Y. Zhao, 
``The point insertion technique and open $r$-spin theories I: moduli and orientation,"
arXiv:2310.13185.

\bibitem[TZ2]{TZ2}  R. J. Tessler and Y. Zhao, 
``The point insertion technique and open $r$-spin theories II: intersection theories in genus-zero,"
arXiv:2311.11779. 

\bibitem[Wa07]{Wa07} J. Walcher, 
``Opening mirror symmetry on the quintic,"
Comm. Math. Phys. {\bf 276} (2007), no. 3, 671--689. 

\bibitem[Wa08]{Wa08} J. Walcher, 
``Open strings and extended mirror symmetry," 
Modular forms and string duality, 279--297, Fields Inst. Commun., {\bf 54}, Amer. Math. Soc., Providence, RI, 2008.

\bibitem[Wa09]{Wa09} J.Walcher,
"Evidence for tadpole cancelation in the topological string,"
Commun. Number Theory Phys. {\bf 3} (2009), no. 1, 111--172.

\bibitem[We05]{We05} J. Y. Welschinger, 
``Invariants of real symplectic 4-manifolds and lower bounds in real enumerative geometry," 
Invent. Math. {\bf 162} (2005), no. 1, 195--234.

\bibitem[We]{We} J. Y. Welschinger, 
``Spinor states of real rational curves in real alge-braic convex 3-manifolds and enumerative invariants,"
Duke Math. J. {\bf 127} (2005), no. 1, 89--121. 

\bibitem[Wi93]{Wi93} E. Witten, 
``Phases of N=2 theories in two dimensions," 
Nuclear Phys. {\bf B 403} (1993), no. 1-2, 159--222.


\end{thebibliography}
\end{document}